\documentclass{article}
\usepackage{amsfonts}
\usepackage{amsmath}
\usepackage{amsthm}
\usepackage[margin=1.0in]{geometry}
\usepackage{graphicx}
\usepackage{tikz}
\usepackage{url}
\usepackage{hyperref}
\usepackage{nicematrix}
\usepackage{multicol}

\newtheoremstyle{custom}
{6pt}% Space above
{3pt}% Space below
{\sl}% Body font
{}% Indent amount
{\bf}% Theorem head font
{.}% Punctuation after theorem head 
{.5em}% Space after theorem head
{}% Theorem head spec (can be left empty, meaning `normal')

\theoremstyle{custom}
\newtheorem{theorem}{Theorem}[section]
\newtheorem{lemma}[theorem]{Lemma}

\newtheorem{corollary}[theorem]{Corollary}
\newtheorem{conjecture}[theorem]{Conjecture}

\newtheorem{statement}[theorem]{Statement}

\title{Further techniques on a polynomial positivity question of Collins, Dykema, and Torres-Ayala}
\author{Nathaniel K.~Green\\Edward D.~Kim\\\sl{University of Wisconsin-La Crosse}}
\begin{document}

\maketitle

\begin{abstract}
We prove that the coefficient of $t^2$ in $\mathsf{trace}((A+tB)^6)$ is a sum of squares in the entries of the symmetric matrices $A$ and $B$.
\end{abstract}


\section{Introduction}

In~\cite{Collins}, Collins, Dykema, and Torres-Ayala asked the following question:
\begin{conjecture}\label{conjecture:CDTA}
If $A$ and $B$ are symmetric matrices of the same size, then for all $m \geq r$ where $m$ and $r$ are even, the coefficient of $t^r$ in $\mathsf{trace}((A+tB)^m)$ is non-negative.
\end{conjecture}
The question was a variant of the following question (in the version stated below, given by Lieb and Seiringer in~\cite{LiebSeiringer}
\begin{conjecture}\label{conjecture:LiebSeiringer}
If $A$ and $B$ are positive semidefinite matrices of the same size, then for all integers $m \geq r \geq 0$, the coefficient of $t^r$ in $\mathsf{trace}((A+tB)^m)$ is non-negative.
\end{conjecture}
This conjecture is neither stronger than nor weaker than Conjecture~\ref{conjecture:CDTA} as the statement relaxes the evenness conditions on $m$ and $r$, but requires $A$ and $B$ to be positive semidefinite. Conjecture~\ref{conjecture:LiebSeiringer} was formulated by Lieb and Seiringer because it is equivalent to the following conjecture from 1975 by Bessis, Moussa, and Villani in~\cite{BessisMoussaVillani}:
\begin{statement}\label{statement:bmv-original}
Given an $n \times n$ Hermitian matrix $A$ and an $n \times n$ positive semidefinite matrix $B$, the function $t \mapsto \mathsf{trace}(\mathsf{exp}(A-tB))$ is the Laplace transform of a positive measure supported in $[0, \infty)$.
\end{statement}
This was proved by Stahl (see~\cite{Stahl:ProofBMV}) and simplified by Eremenko (see~\cite{Eremenko}). Collins, Dykema, and Torres-Ayala (see~\cite{Collins}) and Burgdorf et al.{} (see~\cite{Burgdorf:TracialMoment, Burgdorf:CyclicEquivalenceChapter}) prove the special case of Conjecture~\ref{conjecture:CDTA} when $(m,r)=(4,2)$ by expressing the coefficient of $t^r$ in $\mathsf{trace}((A+tB)^m)$ as a sum of squares of non-commutative variables $A$ and $B$.

In~\cite{KimMillerZinnel:CDTA}, the special case of Conjecture~\ref{conjecture:CDTA} when $(m,r)=(4,2)$ is proved by expressing the coefficient of $t^r$ in $\mathsf{trace}((A+tB)^m)$ as a sum of squares in the $2n + 2\binom{n}{2}$ commutative variables $a_{\{i,j\}}$ and $b_{\{i,j\}}$, the entries of $A$ and $B$. The article~\cite{KimMillerZinnel:CDTA} also examines the special case of Conjecture~\ref{conjecture:CDTA} when $(m,r)=(8,4)$ for small $n$. The focus of~\cite{KimMillerZinnel:CDTA} was on cases of Conjecture~\ref{conjecture:CDTA} when $r$ is a power of $2$ and $m=2r$. The first natural case this leaves out is $(m,r)=(6,2)$, which we examine in Theorem~\ref{theorem:main}.

\begin{theorem}\label{theorem:main}
If $A$ and $B$ are $n \times n$ symmetric matrices, then the coefficient of $t^2$ in $\mathsf{trace}((A+tB)^6)$ is a sum of squares of polynomials in the variables $a_{\{i,j\}}$ and $b_{\{i,j\}}$: there exist positive semidefinite matrices $U$ and $R$ and vectors $y$ and $z_{(i,j)}$ such that the coefficient of $t^2$ in $\mathsf{trace}((A+tB)^6)$ is equal to
\begin{equation}\label{equation:sos}
y^TUy + \sum_{1 \leq i < j \leq n} z_{(i,j)}^TRz_{(i,j)}.
\end{equation}
\end{theorem}

In our proof of Theorem 2.1 we use well known results (see, e.g.,~\cite{Laurent, Prajna, Riener}) to show, given appropriate matrices and vectors, $\mathsf{trace}((A+tB)^6)$ is non-negative. In our proof we will use two matrices $U$ and $R$, along with vectors $y$ and $z_{(i,j)}$ to prove each term from the expression above, $y^TUy,\: z_{(1,2)}^TRz_{(1,2)},\ldots ,\: z_{(n-1,n)}^TRz_{(n-1,n)}$ is non-negative, therefore proving the whole expression non-negative. 

Three large questions arise that will be answered in the following three sections.
We set out to provide index based descriptions of all previously mentioned matrices and vectors in Section~\ref{section:defMatrices&Vecs}. In Section~\ref{section:relevance} we prove $y^TUy + \sum_{1 \leq i < j \leq n} z_{(i,j)}^TRz_{(i,j)}$ equals the coefficient of $t^2$ in $\mathsf{trace}((A+tB)^6)$ by exhaustion. In Section~\ref{section:psdRn} we provide a proof that $R$ is positive semidefinite using the generalized Schur complement and other various methods. In Section~\ref{section:psdU} we provide descriptions of $U, U^2, U^3,$ and $U^4$ to show $U$ satisfies a polynomial $\mu$, which is a multiple of the minimum polynomial of $U$. We then use Descartes' Law of Signs so show $\mu$ has no negative roots, therefore proving $U$ is positive semidefinite. %%%% actually, for Section 3, it's the coeff if t^2 in ...

\subsection{Notation for block constant matrices}\label{section:block-constant-matrices}

When defining the matrix $R$ and subsequent analysis, it will be helpful to introduce notation for block matrices which have constant value within each block. Throughout this paper, whenever a non-negative integer appears outside the presentation of what appears to be a small matrix, such a row heading (respectively, column heading) indicates the number of rows (respectively, columns) in a block, while the entry within the brackets indicates the entry throughout the block. For example,
\[
\begin{bNiceArray}{cc}[first-row,first-col]
& 3 & 4  \\
1 & 5 & 6 \\
2 & 7 & 8
\end{bNiceArray}
=
\left[
\begin{array}{ccccccc}
5 & 5 & 5 & 6 & 6 & 6 & 6 \\
7 & 7 & 7 & 8 & 8 & 8 & 8 \\
7 & 7 & 7 & 8 & 8 & 8 & 8 
\end{array}
\right]
.\]

\section{Definitions of relevant matrices and vectors}\label{section:defMatrices&Vecs}

Throughout the following section we define vectors $y$ and $z_{(i,j)}$, and matrices $U$ and $R$.

\subsection{Definition of $y$}
In this section, we define a vector $y$ of size $3 \cdot \binom{n}{3} + n(n-1)+n^2$. The vector $y$ is partitioned into three blocks. The first block has $3\cdot \binom{n}{3}$ entries indexed by all possible ordered pairs of the form $(T,\ell)$ where $T$ is a $3$-element subset of $[n]$ and $\ell \in T$. The monomial at index $(\{i,j,k\},\ell)$ is $a_{\{\ell,j\}}a_{\{\ell,k\}}b_{\{j,k\}}$. 
The second block of the vector $y$ has $n(n-1)$ entries indexed by all possible ordered pairs $(i,j)$ such that $i\neq j$ and $i,j\leq n$. The monomial at index $(i,j)$ is $a_{\{i,j\}}a_{\{i,i\}}b_{\{i,j\}}$.
The third block of the vector $y$ has $n^2$ entries indexed by all possible ordered pairs $(\hat{i}, \hat{j})$, where $\hat{i},\hat{j}\leq n$. The monomial at index $(\hat{i}, \hat{j})$ is $a_{\{\hat{i},\hat{j}\}}a_{\{\hat{i},\hat{j}\}}b_{\{\hat{i},\hat{i}\}}$.

\subsection{Definition of $U$}\label{section:defU}

We now define a symmetric block matrix $U$ with $3 \cdot \binom{n}{3} + n(n-1)+n^2$ rows and columns, of the form
\[ U = \left[\begin{array}{ccc}
S_1 & S_{12} & S_{13} \\
{S_{12}}^T & S_2 & S_{23} \\
{S_{13}}^T & {S_{23}}^T & S_3
\end{array}
\right]
.\]

Six submatrices are described below, with the remaining three blocks below the block diagonal automatically given by transposition. Each submatrix is indexed the same way as the vector $y$. Moreover, the blocks in $y$ determine the submatrices of $U$. 

\subsubsection{The $S_1$ submatrix}

The entry in the $(\{i,j,k\},\ell)$-row $(\{i',j',k'\},\ell')$-column of $S_1$ is $6(|\{i,j,k\} \cap \{i',j',k'\}| + |\{\ell\} \cap \{\ell'\}|)$, where $\ell \in \{i,j,k\}$ and $\ell' \in \{i',j',k'\}$.

\subsubsection{The $S_2$ submatrix}

The entry in the $(i,j)$-row $(i',j')$-column of $S_2$ is $6|\{j\}\cap \{j'\}|+30|\{i\}\cap \{i'\}|+12|\{i\}\cap \{j'\}|+12|\{j\}\cap \{i'\}|$.

\subsubsection{The $S_3$ submatrix}

The entry in the $(\hat{i},\hat{j})$-row $(\hat{i}',\hat{j}')$-column of $S_3$ is:
\begin{itemize}
\item $3(|\{\hat{i},\hat{j}\} \cap \{\hat{i}',\hat{j}'\}| + |\{\hat{i}\} \cap \{\hat{i}'\}|)$ if $\hat{i} \ne \hat{j}$ and $\hat{i}' \ne \hat{j}'$
\item $3(3|\{\hat{i}\} \cap \{\hat{i}'\}| + 2|\{\hat{j}\} \cap \{\hat{j}'\}|)$ if $\hat{i} =\hat{j}$ or $\hat{i}'= \hat{j}'$
\end{itemize}

\subsubsection{The $S_{12}$ submatrix}

The rows of $S_{12}$ are indexed as in $S_1$, with a typical index being of the form $(\{i,j,k\},\ell)$, where $\ell \in \{i,j,k\}$. The columns of $S_{12}$ are indexed as in $S_2$ with a typical index being of the form $(i',j')$ with $i' \not= j'$.
The entry in row $(\{i,j,k\},\ell)$ column $(i',j')$ is $12|\{i'\} \cap \{i,j,k\}| + 6|\{j'\} \cap \{i,j,k\}| + 6|\{i'\} \cap \{\ell\}|$.

\subsubsection{The $S_{13}$ submatrix}

The rows of $S_{13}$ are indexed as in $S_1$, with a typical index being of the form $(\{i,j,k\},\ell)$, where $\ell \in \{i,j,k\}$. The columns of $S_{13}$ are indexed as in $S_2$ with a typical index being of the form $(\hat{i},\hat{j})$. The entry in row $(\{i,j,k\},\ell)$ column $(\hat{i},\hat{j})$ is $6|\{\hat{i}\} \cap \{i,j,k\}| + 3|\{\hat{j}\} \cap \{i,j,k\}| + 3|\{\hat{j}\} \cap \{\ell\}|$.

\subsubsection{The $S_{23}$ submatrix}

The rows of $S_{23}$ are indexed as in $S_2$, with a typical index being of the form $(i,j)$ with $i \not= j$. The columns of $S_{23}$ are indexed as in $S_3$, with a typical index being of the form $(\hat{i},\hat{j})$. The entry in the $(i,j)$ row $(\hat{i},\hat{j})$ column of $S_{23}$ is:
\begin{itemize}
\item $15$ if $i=\hat{i}$ and $j=\hat{j}$
\item Otherwise, $21$ if $i=\hat{i}$ and $i=\hat{j}$
\item Otherwise, $12$ if $i=\hat{i}$ and $i \ne j$ and $j\ne \hat{i}$ and $i\ne \hat{j}$
\item Otherwise, $15$ if $j=\hat{i}$ and $i=\hat{j}$
\item Otherwise, $9$ if $i=\hat{j}$ and $j \ne \hat{i}$ and $i \ne \hat{i}$
\item Otherwise, $9$ if $\hat{i} = \hat{j}$ and $j= \hat{i}$
\item Otherwise, $6$ if $j= \hat{i}$ and $i \neq \hat{j}$ and $j \neq \hat{j}$
\item Otherwise, $3$ if $j = \hat{j}$ and $\hat{i} \neq \hat{j}$
\item Otherwise $0$.
\end{itemize}

\subsection{Definition of $z_{(i,j)}$}

Let $z_{(i,j)}$ be a vector of size $3n^2 - 2n$, broken up into six blocks as described below.
\begin{itemize}
    \item The first block contains a single monomial $a_{\{i,j\}}^2b_{\{i,j\}}$, with index $(1)$.

    \item The second block has size $n-1$, and consists of the single monomial $a_{\{i,i\}}a_{\{i,j\}}b_{\{i,i\}}$ with index $(2,1)$, followed by $n-2$ monomials of the form $a_{\{i,j\}}a_{\{i,k\}}b_{\{i,k\}}$ with index $(2,2,k)$, where $k\leq n$ and $k \neq i,j$.

    \item The third block has the $n-1$ monomials of the form $a_{\{i,j\}}a_{\{j,k\}}b_{\{j,k\}}$ with index $(3,k)$, where $k\leq n$ and $k\neq i$.

    \item The fourth block has size $(n-1)^2$, and consists of the single monomial $a_{\{i,i\}}a_{\{j,j\}}b_{\{i,j\}}$ with index $(4,1)$, followed by $n-2$ monomials of the form $a_{\{i,i\}}a_{\{j,k\}}b_{\{i,k\}}$ with index $(4,2,k)$, then $n-2$ monomials of the form $a_{\{i,k\}}a_{\{j,j\}}b_{\{j,k\}}$ with index $(4,3,k)$, then $n-2$ monomials of the form $a_{\{i,k\}}a_{\{j,k\}}b_{\{k,k\}}$ with index $(4,4,k)$. Finally, $n^2 - 5n + 6$ monomials of the form $a_{\{i,k\}}a_{\{j,l\}}b_{\{k,l\}}$ with index $(4,5,k,l)$, where $k\neq i,j,l$ and $l\neq i,j,k$. 

    \item The fifth block has size $n(n-1)$, starting with monomials $a_{\{i,i\}}a_{\{i,i\}}b_{\{i,j\}}$ with index $(5,1)$ and $a_{\{i,i\}}a_{\{i,j\}}b_{\{j,j\}}$ with index $(5,2)$. Then follow $n-2$ monomials of the form $a_{\{i,i\}}a_{\{i,k\}}b_{\{k,j\}}$ with index $(5,3,k)$, then $n-2$ monomials of the form $a_{\{i,k\}}a_{\{i,k\}}b_{\{i,j\}}$ with index $(5,4,k)$, then $n-2$ monomials of the form $a_{\{i,k\}}a_{\{j,k\}}b_{\{j,j\}}$ with index $(5,5,k)$, and $n-2$ monomials of the form $a_{\{i,k\}}a_{\{k,k\}}b_{\{j,k\}}$ with index $(5,6,k)$, where $k \neq i,j$. Finally, there are $n^2 - 5n + 6$ monomials of the form $a_{\{i,k\}}a_{\{k,l\}}b_{\{j,l\}}$ with index $(5,7,k,l)$, where $k\neq i,j,l$ and $l\neq i,j,k$. 

    \item The sixth block has size $n(n-1)$, starting with monomials $a_{\{j,j\}}a_{\{i,j\}}b_{\{i,i\}}$ with index $(6,1)$ and $a_{\{j,j\}}a_{\{j,j\}}b_{\{i,j\}}$ with index $(6,2)$. Then follow $n-2$ monomials of the form $a_{\{j,j\}}a_{\{k,j\}}b_{\{i,k\}}$ with index $(6,3,k)$, then $n-2$ monomials of the form $a_{\{j,k\}}a_{\{i,k\}}b_{\{i,i\}}$ with index $(6,4,k)$, then $n-2$ monomials of the form $a_{\{j,k\}}a_{\{j,k\}}b_{\{i,j\}}$ with index $(6,5,k)$, and $n-2$ monomials of the form $a_{\{j,k\}}a_{\{k,k\}}b_{\{i,k\}}$ with index $(6,6,k)$, where $k \neq i,j$. Finally, there are $n^2 - 5n + 6$ monomials of the form $a_{\{j,k\}}a_{\{k,l\}}b_{\{i,l\}}$ with index $(6,7,k,l)$, where $k\neq i,j,l$ and $l\neq i,j,k$. 

\end{itemize}
For our use of the index language, in what follows, we will write, for example, the $(2,2,3)$-entry of $z_{(1,2)}$ is the monomial $a_{\{1,2\}}a_{\{1,3\}}b_{\{2,3\}}$.

\subsection{Definition of $R$}

Just as the indexing of the vector $y$ determined $U$, the blocks of $z_{(i,j)}$ determine $R$. The description of $R$ is simple compared to $U$. Using the notation introduced in Section~\ref{section:block-constant-matrices}, for a fixed $n$, the matrix $R$ is the $(3n^2-2n) \times (3n^2-2n)$ symmetric matrix
\[
\begin{bNiceArray}{cccccc}[first-row,first-col]
& 1 & n-1 & n-1 & (n-1)^2 & n(n-1) & n(n-1) \\
1 & 30& 21& 21& 12& 9& 9\\
n-1 & 21& 18& 12& 9& 3& 9\\
n-1 & 21& 12& 18& 9& 9& 3\\
(n-1)^2 & 12& 9& 9& 6& 3& 3\\
n(n-1) & 9& 3& 9& 3& 6& 0\\
n(n-1) & 9& 9& 3& 3& 0& 6
\end{bNiceArray}
.
\]

\newpage

\subsection{Examples of matrices and vectors for $n=4$}

\begin{multicols}{2}

$$y=\left[\begin{array}{cccccccccccccccccccccccccccccccccccccccc}a_{\{1,2\}} a_{\{1,3\}} b_{\{2,3\}}\\     a_{\{1,2\}} a_{\{2,3\}} b_{\{1,3\}}\\     a_{\{1,3\}} a_{\{2,3\}} b_{\{1,2\}}\\     a_{\{1,2\}} a_{\{1,4\}} b_{\{2,4\}}\\     a_{\{1,2\}} a_{\{2,4\}} b_{\{1,4\}}\\     a_{\{1,4\}} a_{\{2,4\}} b_{\{1,2\}}\\     a_{\{1,3\}} a_{\{1,4\}} b_{\{3,4\}}\\     a_{\{1,3\}} a_{\{3,4\}} b_{\{1,4\}}\\     a_{\{1,4\}} a_{\{3,4\}} b_{\{1,3\}}\\     a_{\{2,3\}} a_{\{2,4\}} b_{\{3,4\}}\\     a_{\{2,3\}} a_{\{3,4\}} b_{\{2,4\}}\\     a_{\{2,4\}} a_{\{3,4\}} b_{\{2,3\}}\\     a_{\{1,1\}} a_{\{1,2\}} b_{\{1,2\}}\\     a_{\{1,1\}} a_{\{1,3\}} b_{\{1,3\}}\\     a_{\{1,1\}} a_{\{1,4\}} b_{\{1,4\}}\\     a_{\{1,2\}} a_{\{2,2\}} b_{\{1,2\}}\\     a_{\{2,2\}} a_{\{2,3\}} b_{\{2,3\}}\\     a_{\{2,2\}} a_{\{2,4\}} b_{\{2,4\}}\\     a_{\{1,3\}} a_{\{3,3\}} b_{\{1,3\}}\\     a_{\{2,3\}} a_{\{3,3\}} b_{\{2,3\}}\\     a_{\{3,3\}} a_{\{3,4\}} b_{\{3,4\}}\\     a_{\{1,4\}} a_{\{4,4\}} b_{\{1,4\}}\\     a_{\{2,4\}} a_{\{4,4\}} b_{\{2,4\}}\\     a_{\{3,4\}} a_{\{4,4\}} b_{\{3,4\}}\\     a_{\{1,1\}}^{2} b_{\{1,1\}}\\     a_{\{1,2\}}^{2} b_{\{1,1\}}\\     a_{\{1,3\}}^{2} b_{\{1,1\}}\\     a_{\{1,4\}}^{2} b_{\{1,1\}}\\     a_{\{1,2\}}^{2} b_{\{2,2\}}\\     a_{\{2,2\}}^{2} b_{\{2,2\}}\\     a_{\{2,3\}}^{2} b_{\{2,2\}}\\     a_{\{2,4\}}^{2} b_{\{2,2\}}\\     a_{\{1,3\}}^{2} b_{\{3,3\}}\\     a_{\{2,3\}}^{2} b_{\{3,3\}}\\     a_{\{3,3\}}^{2} b_{\{3,3\}}\\     a_{\{3,4\}}^{2} b_{\{3,3\}}\\     a_{\{1,4\}}^{2} b_{\{4,4\}}\\     a_{\{2,4\}}^{2} b_{\{4,4\}}\\     a_{\{3,4\}}^{2} b_{\{4,4\}}\\     a_{\{4,4\}}^{2} b_{\{4,4\}}\end{array}\right]$$

\columnbreak

$$z_{(1,2)} = \left[\begin{array}{cccccccccccccccccccccccccccccccccccccccc}a_{\{1,2\}}^{2} b_{\{1,2\}}\\     a_{\{1,1\}} a_{\{1,2\}} b_{\{1,1\}}\\     a_{\{1,2\}} a_{\{1,3\}} b_{\{1,3\}}\\     a_{\{1,2\}} a_{\{1,4\}} b_{\{1,4\}}\\     a_{\{1,2\}} a_{\{2,2\}} b_{\{2,2\}}\\     a_{\{1,2\}} a_{\{2,3\}} b_{\{2,3\}}\\     a_{\{1,2\}} a_{\{2,4\}} b_{\{2,4\}}\\     a_{\{1,1\}} a_{\{2,2\}} b_{\{1,2\}}\\     a_{\{1,1\}} a_{\{2,3\}} b_{\{1,3\}}\\     a_{\{1,1\}} a_{\{2,4\}} b_{\{1,4\}}\\     a_{\{1,3\}} a_{\{2,2\}} b_{\{2,3\}}\\     a_{\{1,4\}} a_{\{2,2\}} b_{\{2,4\}}\\     a_{\{1,3\}} a_{\{2,3\}} b_{\{3,3\}}\\     a_{\{1,4\}} a_{\{2,4\}} b_{\{4,4\}}\\     a_{\{1,3\}} a_{\{2,4\}} b_{\{3,4\}}\\     a_{\{1,4\}} a_{\{2,3\}} b_{\{3,4\}}\\     a_{\{1,1\}}^{2} b_{\{1,2\}}\\     a_{\{1,1\}} a_{\{1,2\}} b_{\{2,2\}}\\     a_{\{1,1\}} a_{\{1,3\}} b_{\{2,3\}}\\     a_{\{1,1\}} a_{\{1,4\}} b_{\{2,4\}}\\     a_{\{1,3\}}^{2} b_{\{1,2\}}\\     a_{\{1,4\}}^{2} b_{\{1,2\}}\\     a_{\{1,3\}} a_{\{2,3\}} b_{\{2,2\}}\\     a_{\{1,4\}} a_{\{2,4\}} b_{\{2,2\}}\\     a_{\{1,3\}} a_{\{3,3\}} b_{\{2,3\}}\\     a_{\{1,4\}} a_{\{4,4\}} b_{\{2,4\}}\\     a_{\{1,3\}} a_{\{3,4\}} b_{\{2,4\}}\\     a_{\{1,4\}} a_{\{3,4\}} b_{\{2,3\}}\\     a_{\{1,2\}} a_{\{2,2\}} b_{\{1,1\}}\\     a_{\{2,2\}}^{2} b_{\{1,2\}}\\     a_{\{2,2\}} a_{\{2,3\}} b_{\{1,3\}}\\     a_{\{2,2\}} a_{\{2,4\}} b_{\{1,4\}}\\     a_{\{1,3\}} a_{\{2,3\}} b_{\{1,1\}}\\     a_{\{1,4\}} a_{\{2,4\}} b_{\{1,1\}}\\     a_{\{2,3\}}^{2} b_{\{1,2\}}\\     a_{\{2,4\}}^{2} b_{\{1,2\}}\\     a_{\{2,3\}} a_{\{3,3\}} b_{\{1,3\}}\\     a_{\{2,4\}} a_{\{4,4\}} b_{\{1,4\}}\\     a_{\{2,3\}} a_{\{3,4\}} b_{\{1,4\}}\\     a_{\{2,4\}} a_{\{3,4\}} b_{\{1,3\}}\end{array}\right]$$

\end{multicols}

{\setlength{\arraycolsep}{.05cm} $$U = \tiny \left[\begin{array}{cccccccccccccccccccccccccccccccccccccccc}
24 & 18 & 18 & 18 & 12 & 12 & 18 & 12 & 12 & 12 & 12 & 12 & 24 & 24 & 18 & 18 & 18 & 12 & 18 & 18 & 12 & 6 & 6 & 6 & 12 & 9 & 9 & 6 & 12 & 9 & 9 & 6 & 12 & 9 & 9 & 6 & 6 & 3 & 3 & 0 \\
18 & 24 & 18 & 12 & 18 & 12 & 12 & 12 & 12 & 18 & 12 & 12 & 18 & 18 & 12 & 24 & 24 & 18 & 18 & 18 & 12 & 6 & 6 & 6 & 9 & 12 & 9 & 6 & 9 & 12 & 9 & 6 & 9 & 12 & 9 & 6 & 3 & 6 & 3 & 0 \\
18 & 18 & 24 & 12 & 12 & 12 & 12 & 18 & 12 & 12 & 18 & 12 & 18 & 18 & 12 & 18 & 18 & 12 & 24 & 24 & 18 & 6 & 6 & 6 & 9 & 9 & 12 & 6 & 9 & 9 & 12 & 6 & 9 & 9 & 12 & 6 & 3 & 3 & 6 & 0 \\
18 & 12 & 12 & 24 & 18 & 18 & 18 & 12 & 12 & 12 & 12 & 12 & 24 & 18 & 24 & 18 & 12 & 18 & 6 & 6 & 6 & 18 & 18 & 12 & 12 & 9 & 6 & 9 & 12 & 9 & 6 & 9 & 6 & 3 & 0 & 3 & 12 & 9 & 6 & 9 \\
12 & 18 & 12 & 18 & 24 & 18 & 12 & 12 & 12 & 18 & 12 & 12 & 18 & 12 & 18 & 24 & 18 & 24 & 6 & 6 & 6 & 18 & 18 & 12 & 9 & 12 & 6 & 9 & 9 & 12 & 6 & 9 & 3 & 6 & 0 & 3 & 9 & 12 & 6 & 9 \\
12 & 12 & 12 & 18 & 18 & 24 & 12 & 12 & 18 & 12 & 12 & 18 & 18 & 12 & 18 & 18 & 12 & 18 & 6 & 6 & 6 & 24 & 24 & 18 & 9 & 9 & 6 & 12 & 9 & 9 & 6 & 12 & 3 & 3 & 0 & 6 & 9 & 9 & 6 & 12 \\
18 & 12 & 12 & 18 & 12 & 12 & 24 & 18 & 18 & 12 & 12 & 12 & 18 & 24 & 24 & 6 & 6 & 6 & 18 & 12 & 18 & 18 & 12 & 18 & 12 & 6 & 9 & 9 & 6 & 0 & 3 & 3 & 12 & 6 & 9 & 9 & 12 & 6 & 9 & 9 \\
12 & 12 & 18 & 12 & 12 & 12 & 18 & 24 & 18 & 12 & 18 & 12 & 12 & 18 & 18 & 6 & 6 & 6 & 24 & 18 & 24 & 18 & 12 & 18 & 9 & 6 & 12 & 9 & 3 & 0 & 6 & 3 & 9 & 6 & 12 & 9 & 9 & 6 & 12 & 9 \\
12 & 12 & 12 & 12 & 12 & 18 & 18 & 18 & 24 & 12 & 12 & 18 & 12 & 18 & 18 & 6 & 6 & 6 & 18 & 12 & 18 & 24 & 18 & 24 & 9 & 6 & 9 & 12 & 3 & 0 & 3 & 6 & 9 & 6 & 9 & 12 & 9 & 6 & 9 & 12 \\
12 & 18 & 12 & 12 & 18 & 12 & 12 & 12 & 12 & 24 & 18 & 18 & 6 & 6 & 6 & 18 & 24 & 24 & 12 & 18 & 18 & 12 & 18 & 18 & 0 & 6 & 3 & 3 & 6 & 12 & 9 & 9 & 6 & 12 & 9 & 9 & 6 & 12 & 9 & 9 \\
12 & 12 & 18 & 12 & 12 & 12 & 12 & 18 & 12 & 18 & 24 & 18 & 6 & 6 & 6 & 12 & 18 & 18 & 18 & 24 & 24 & 12 & 18 & 18 & 0 & 3 & 6 & 3 & 6 & 9 & 12 & 9 & 6 & 9 & 12 & 9 & 6 & 9 & 12 & 9 \\
12 & 12 & 12 & 12 & 12 & 18 & 12 & 12 & 18 & 18 & 18 & 24 & 6 & 6 & 6 & 12 & 18 & 18 & 12 & 18 & 18 & 18 & 24 & 24 & 0 & 3 & 3 & 6 & 6 & 9 & 9 & 12 & 6 & 9 & 9 & 12 & 6 & 9 & 9 & 12 \\
24 & 18 & 18 & 24 & 18 & 18 & 18 & 12 & 12 & 6 & 6 & 6 & 36 & 30 & 30 & 24 & 12 & 12 & 12 & 6 & 0 & 12 & 6 & 0 & 21 & 15 & 12 & 12 & 15 & 9 & 6 & 6 & 9 & 3 & 0 & 0 & 9 & 3 & 0 & 0 \\
24 & 18 & 18 & 18 & 12 & 12 & 24 & 18 & 18 & 6 & 6 & 6 & 30 & 36 & 30 & 12 & 6 & 0 & 24 & 12 & 12 & 12 & 0 & 6 & 21 & 12 & 15 & 12 & 9 & 0 & 3 & 0 & 15 & 6 & 9 & 6 & 9 & 0 & 3 & 0 \\
18 & 12 & 12 & 24 & 18 & 18 & 24 & 18 & 18 & 6 & 6 & 6 & 30 & 30 & 36 & 12 & 0 & 6 & 12 & 0 & 6 & 24 & 12 & 12 & 21 & 12 & 12 & 15 & 9 & 0 & 0 & 3 & 9 & 0 & 0 & 3 & 15 & 6 & 6 & 9 \\
18 & 24 & 18 & 18 & 24 & 18 & 6 & 6 & 6 & 18 & 12 & 12 & 24 & 12 & 12 & 36 & 30 & 30 & 6 & 12 & 0 & 6 & 12 & 0 & 9 & 15 & 6 & 6 & 15 & 21 & 12 & 12 & 3 & 9 & 0 & 0 & 3 & 9 & 0 & 0 \\
18 & 24 & 18 & 12 & 18 & 12 & 6 & 6 & 6 & 24 & 18 & 18 & 12 & 6 & 0 & 30 & 36 & 30 & 12 & 24 & 12 & 0 & 12 & 6 & 0 & 9 & 3 & 0 & 12 & 21 & 15 & 12 & 6 & 15 & 9 & 6 & 0 & 9 & 3 & 0 \\
12 & 18 & 12 & 18 & 24 & 18 & 6 & 6 & 6 & 24 & 18 & 18 & 12 & 0 & 6 & 30 & 30 & 36 & 0 & 12 & 6 & 12 & 24 & 12 & 0 & 9 & 0 & 3 & 12 & 21 & 12 & 15 & 0 & 9 & 0 & 3 & 6 & 15 & 6 & 9 \\
18 & 18 & 24 & 6 & 6 & 6 & 18 & 24 & 18 & 12 & 18 & 12 & 12 & 24 & 12 & 6 & 12 & 0 & 36 & 30 & 30 & 6 & 0 & 12 & 9 & 6 & 15 & 6 & 3 & 0 & 9 & 0 & 15 & 12 & 21 & 12 & 3 & 0 & 9 & 0 \\
18 & 18 & 24 & 6 & 6 & 6 & 12 & 18 & 12 & 18 & 24 & 18 & 6 & 12 & 0 & 12 & 24 & 12 & 30 & 36 & 30 & 0 & 6 & 12 & 0 & 3 & 9 & 0 & 6 & 9 & 15 & 6 & 12 & 15 & 21 & 12 & 0 & 3 & 9 & 0 \\
12 & 12 & 18 & 6 & 6 & 6 & 18 & 24 & 18 & 18 & 24 & 18 & 0 & 12 & 6 & 0 & 12 & 6 & 30 & 30 & 36 & 12 & 12 & 24 & 0 & 0 & 9 & 3 & 0 & 0 & 9 & 3 & 12 & 12 & 21 & 15 & 6 & 6 & 15 & 9 \\
6 & 6 & 6 & 18 & 18 & 24 & 18 & 18 & 24 & 12 & 12 & 18 & 12 & 12 & 24 & 6 & 0 & 12 & 6 & 0 & 12 & 36 & 30 & 30 & 9 & 6 & 6 & 15 & 3 & 0 & 0 & 9 & 3 & 0 & 0 & 9 & 15 & 12 & 12 & 21 \\
6 & 6 & 6 & 18 & 18 & 24 & 12 & 12 & 18 & 18 & 18 & 24 & 6 & 0 & 12 & 12 & 12 & 24 & 0 & 6 & 12 & 30 & 36 & 30 & 0 & 3 & 0 & 9 & 6 & 9 & 6 & 15 & 0 & 3 & 0 & 9 & 12 & 15 & 12 & 21 \\
6 & 6 & 6 & 12 & 12 & 18 & 18 & 18 & 24 & 18 & 18 & 24 & 0 & 6 & 12 & 0 & 6 & 12 & 12 & 12 & 24 & 30 & 30 & 36 & 0 & 0 & 3 & 9 & 0 & 0 & 3 & 9 & 6 & 6 & 9 & 15 & 12 & 12 & 15 & 21 \\
12 & 9 & 9 & 12 & 9 & 9 & 12 & 9 & 9 & 0 & 0 & 0 & 21 & 21 & 21 & 9 & 0 & 0 & 9 & 0 & 0 & 9 & 0 & 0 & 15 & 9 & 9 & 9 & 6 & 0 & 0 & 0 & 6 & 0 & 0 & 0 & 6 & 0 & 0 & 0 \\
9 & 12 & 9 & 9 & 12 & 9 & 6 & 6 & 6 & 6 & 3 & 3 & 15 & 12 & 12 & 15 & 9 & 9 & 6 & 3 & 0 & 6 & 3 & 0 & 9 & 9 & 6 & 6 & 6 & 6 & 3 & 3 & 3 & 3 & 0 & 0 & 3 & 3 & 0 & 0 \\
9 & 9 & 12 & 6 & 6 & 6 & 9 & 12 & 9 & 3 & 6 & 3 & 12 & 15 & 12 & 6 & 3 & 0 & 15 & 9 & 9 & 6 & 0 & 3 & 9 & 6 & 9 & 6 & 3 & 0 & 3 & 0 & 6 & 3 & 6 & 3 & 3 & 0 & 3 & 0 \\
6 & 6 & 6 & 9 & 9 & 12 & 9 & 9 & 12 & 3 & 3 & 6 & 12 & 12 & 15 & 6 & 0 & 3 & 6 & 0 & 3 & 15 & 9 & 9 & 9 & 6 & 6 & 9 & 3 & 0 & 0 & 3 & 3 & 0 & 0 & 3 & 6 & 3 & 3 & 6 \\
12 & 9 & 9 & 12 & 9 & 9 & 6 & 3 & 3 & 6 & 6 & 6 & 15 & 9 & 9 & 15 & 12 & 12 & 3 & 6 & 0 & 3 & 6 & 0 & 6 & 6 & 3 & 3 & 9 & 9 & 6 & 6 & 3 & 3 & 0 & 0 & 3 & 3 & 0 & 0 \\
9 & 12 & 9 & 9 & 12 & 9 & 0 & 0 & 0 & 12 & 9 & 9 & 9 & 0 & 0 & 21 & 21 & 21 & 0 & 9 & 0 & 0 & 9 & 0 & 0 & 6 & 0 & 0 & 9 & 15 & 9 & 9 & 0 & 6 & 0 & 0 & 0 & 6 & 0 & 0 \\
9 & 9 & 12 & 6 & 6 & 6 & 3 & 6 & 3 & 9 & 12 & 9 & 6 & 3 & 0 & 12 & 15 & 12 & 9 & 15 & 9 & 0 & 6 & 3 & 0 & 3 & 3 & 0 & 6 & 9 & 9 & 6 & 3 & 6 & 6 & 3 & 0 & 3 & 3 & 0 \\
6 & 6 & 6 & 9 & 9 & 12 & 3 & 3 & 6 & 9 & 9 & 12 & 6 & 0 & 3 & 12 & 12 & 15 & 0 & 6 & 3 & 9 & 15 & 9 & 0 & 3 & 0 & 3 & 6 & 9 & 6 & 9 & 0 & 3 & 0 & 3 & 3 & 6 & 3 & 6 \\
12 & 9 & 9 & 6 & 3 & 3 & 12 & 9 & 9 & 6 & 6 & 6 & 9 & 15 & 9 & 3 & 6 & 0 & 15 & 12 & 12 & 3 & 0 & 6 & 6 & 3 & 6 & 3 & 3 & 0 & 3 & 0 & 9 & 6 & 9 & 6 & 3 & 0 & 3 & 0 \\
9 & 12 & 9 & 3 & 6 & 3 & 6 & 6 & 6 & 12 & 9 & 9 & 3 & 6 & 0 & 9 & 15 & 9 & 12 & 15 & 12 & 0 & 3 & 6 & 0 & 3 & 3 & 0 & 3 & 6 & 6 & 3 & 6 & 9 & 9 & 6 & 0 & 3 & 3 & 0 \\
9 & 9 & 12 & 0 & 0 & 0 & 9 & 12 & 9 & 9 & 12 & 9 & 0 & 9 & 0 & 0 & 9 & 0 & 21 & 21 & 21 & 0 & 0 & 9 & 0 & 0 & 6 & 0 & 0 & 0 & 6 & 0 & 9 & 9 & 15 & 9 & 0 & 0 & 6 & 0 \\
6 & 6 & 6 & 3 & 3 & 6 & 9 & 9 & 12 & 9 & 9 & 12 & 0 & 6 & 3 & 0 & 6 & 3 & 12 & 12 & 15 & 9 & 9 & 15 & 0 & 0 & 3 & 3 & 0 & 0 & 3 & 3 & 6 & 6 & 9 & 9 & 3 & 3 & 6 & 6 \\
6 & 3 & 3 & 12 & 9 & 9 & 12 & 9 & 9 & 6 & 6 & 6 & 9 & 9 & 15 & 3 & 0 & 6 & 3 & 0 & 6 & 15 & 12 & 12 & 6 & 3 & 3 & 6 & 3 & 0 & 0 & 3 & 3 & 0 & 0 & 3 & 9 & 6 & 6 & 9 \\
3 & 6 & 3 & 9 & 12 & 9 & 6 & 6 & 6 & 12 & 9 & 9 & 3 & 0 & 6 & 9 & 9 & 15 & 0 & 3 & 6 & 12 & 15 & 12 & 0 & 3 & 0 & 3 & 3 & 6 & 3 & 6 & 0 & 3 & 0 & 3 & 6 & 9 & 6 & 9 \\
3 & 3 & 6 & 6 & 6 & 6 & 9 & 12 & 9 & 9 & 12 & 9 & 0 & 3 & 6 & 0 & 3 & 6 & 9 & 9 & 15 & 12 & 12 & 15 & 0 & 0 & 3 & 3 & 0 & 0 & 3 & 3 & 3 & 3 & 6 & 6 & 6 & 6 & 9 & 9 \\
0 & 0 & 0 & 9 & 9 & 12 & 9 & 9 & 12 & 9 & 9 & 12 & 0 & 0 & 9 & 0 & 0 & 9 & 0 & 0 & 9 & 21 & 21 & 21 & 0 & 0 & 0 & 6 & 0 & 0 & 0 & 6 & 0 & 0 & 0 & 6 & 9 & 9 & 9 & 15
\end{array}\right]$$}

{\setlength{\arraycolsep}{.08cm} $$R = \tiny \left[\begin{array}{cccccccccccccccccccccccccccccccccccccccc}
30 & 21 & 21 & 21 & 21 & 21 & 21 & 12 & 12 & 12 & 12 & 12 & 12 & 12 & 12 & 12 & 9 & 9 & 9 & 9 & 9 & 9 & 9 & 9 & 9 & 9 & 9 & 9 & 9 & 9 & 9 & 9 & 9 & 9 & 9 & 9 & 9 & 9 & 9 & 9 \\
21 & 18 & 18 & 18 & 12 & 12 & 12 & 9 & 9 & 9 & 9 & 9 & 9 & 9 & 9 & 9 & 3 & 3 & 3 & 3 & 3 & 3 & 3 & 3 & 3 & 3 & 3 & 3 & 9 & 9 & 9 & 9 & 9 & 9 & 9 & 9 & 9 & 9 & 9 & 9 \\
21 & 18 & 18 & 18 & 12 & 12 & 12 & 9 & 9 & 9 & 9 & 9 & 9 & 9 & 9 & 9 & 3 & 3 & 3 & 3 & 3 & 3 & 3 & 3 & 3 & 3 & 3 & 3 & 9 & 9 & 9 & 9 & 9 & 9 & 9 & 9 & 9 & 9 & 9 & 9 \\
21 & 18 & 18 & 18 & 12 & 12 & 12 & 9 & 9 & 9 & 9 & 9 & 9 & 9 & 9 & 9 & 3 & 3 & 3 & 3 & 3 & 3 & 3 & 3 & 3 & 3 & 3 & 3 & 9 & 9 & 9 & 9 & 9 & 9 & 9 & 9 & 9 & 9 & 9 & 9 \\
21 & 12 & 12 & 12 & 18 & 18 & 18 & 9 & 9 & 9 & 9 & 9 & 9 & 9 & 9 & 9 & 9 & 9 & 9 & 9 & 9 & 9 & 9 & 9 & 9 & 9 & 9 & 9 & 3 & 3 & 3 & 3 & 3 & 3 & 3 & 3 & 3 & 3 & 3 & 3 \\
21 & 12 & 12 & 12 & 18 & 18 & 18 & 9 & 9 & 9 & 9 & 9 & 9 & 9 & 9 & 9 & 9 & 9 & 9 & 9 & 9 & 9 & 9 & 9 & 9 & 9 & 9 & 9 & 3 & 3 & 3 & 3 & 3 & 3 & 3 & 3 & 3 & 3 & 3 & 3 \\
21 & 12 & 12 & 12 & 18 & 18 & 18 & 9 & 9 & 9 & 9 & 9 & 9 & 9 & 9 & 9 & 9 & 9 & 9 & 9 & 9 & 9 & 9 & 9 & 9 & 9 & 9 & 9 & 3 & 3 & 3 & 3 & 3 & 3 & 3 & 3 & 3 & 3 & 3 & 3 \\
12 & 9 & 9 & 9 & 9 & 9 & 9 & 6 & 6 & 6 & 6 & 6 & 6 & 6 & 6 & 6 & 3 & 3 & 3 & 3 & 3 & 3 & 3 & 3 & 3 & 3 & 3 & 3 & 3 & 3 & 3 & 3 & 3 & 3 & 3 & 3 & 3 & 3 & 3 & 3 \\
12 & 9 & 9 & 9 & 9 & 9 & 9 & 6 & 6 & 6 & 6 & 6 & 6 & 6 & 6 & 6 & 3 & 3 & 3 & 3 & 3 & 3 & 3 & 3 & 3 & 3 & 3 & 3 & 3 & 3 & 3 & 3 & 3 & 3 & 3 & 3 & 3 & 3 & 3 & 3 \\
12 & 9 & 9 & 9 & 9 & 9 & 9 & 6 & 6 & 6 & 6 & 6 & 6 & 6 & 6 & 6 & 3 & 3 & 3 & 3 & 3 & 3 & 3 & 3 & 3 & 3 & 3 & 3 & 3 & 3 & 3 & 3 & 3 & 3 & 3 & 3 & 3 & 3 & 3 & 3 \\
12 & 9 & 9 & 9 & 9 & 9 & 9 & 6 & 6 & 6 & 6 & 6 & 6 & 6 & 6 & 6 & 3 & 3 & 3 & 3 & 3 & 3 & 3 & 3 & 3 & 3 & 3 & 3 & 3 & 3 & 3 & 3 & 3 & 3 & 3 & 3 & 3 & 3 & 3 & 3 \\
12 & 9 & 9 & 9 & 9 & 9 & 9 & 6 & 6 & 6 & 6 & 6 & 6 & 6 & 6 & 6 & 3 & 3 & 3 & 3 & 3 & 3 & 3 & 3 & 3 & 3 & 3 & 3 & 3 & 3 & 3 & 3 & 3 & 3 & 3 & 3 & 3 & 3 & 3 & 3 \\
12 & 9 & 9 & 9 & 9 & 9 & 9 & 6 & 6 & 6 & 6 & 6 & 6 & 6 & 6 & 6 & 3 & 3 & 3 & 3 & 3 & 3 & 3 & 3 & 3 & 3 & 3 & 3 & 3 & 3 & 3 & 3 & 3 & 3 & 3 & 3 & 3 & 3 & 3 & 3 \\
12 & 9 & 9 & 9 & 9 & 9 & 9 & 6 & 6 & 6 & 6 & 6 & 6 & 6 & 6 & 6 & 3 & 3 & 3 & 3 & 3 & 3 & 3 & 3 & 3 & 3 & 3 & 3 & 3 & 3 & 3 & 3 & 3 & 3 & 3 & 3 & 3 & 3 & 3 & 3 \\
12 & 9 & 9 & 9 & 9 & 9 & 9 & 6 & 6 & 6 & 6 & 6 & 6 & 6 & 6 & 6 & 3 & 3 & 3 & 3 & 3 & 3 & 3 & 3 & 3 & 3 & 3 & 3 & 3 & 3 & 3 & 3 & 3 & 3 & 3 & 3 & 3 & 3 & 3 & 3 \\
12 & 9 & 9 & 9 & 9 & 9 & 9 & 6 & 6 & 6 & 6 & 6 & 6 & 6 & 6 & 6 & 3 & 3 & 3 & 3 & 3 & 3 & 3 & 3 & 3 & 3 & 3 & 3 & 3 & 3 & 3 & 3 & 3 & 3 & 3 & 3 & 3 & 3 & 3 & 3 \\
9 & 3 & 3 & 3 & 9 & 9 & 9 & 3 & 3 & 3 & 3 & 3 & 3 & 3 & 3 & 3 & 6 & 6 & 6 & 6 & 6 & 6 & 6 & 6 & 6 & 6 & 6 & 6 & 0 & 0 & 0 & 0 & 0 & 0 & 0 & 0 & 0 & 0 & 0 & 0 \\
9 & 3 & 3 & 3 & 9 & 9 & 9 & 3 & 3 & 3 & 3 & 3 & 3 & 3 & 3 & 3 & 6 & 6 & 6 & 6 & 6 & 6 & 6 & 6 & 6 & 6 & 6 & 6 & 0 & 0 & 0 & 0 & 0 & 0 & 0 & 0 & 0 & 0 & 0 & 0 \\
9 & 3 & 3 & 3 & 9 & 9 & 9 & 3 & 3 & 3 & 3 & 3 & 3 & 3 & 3 & 3 & 6 & 6 & 6 & 6 & 6 & 6 & 6 & 6 & 6 & 6 & 6 & 6 & 0 & 0 & 0 & 0 & 0 & 0 & 0 & 0 & 0 & 0 & 0 & 0 \\
9 & 3 & 3 & 3 & 9 & 9 & 9 & 3 & 3 & 3 & 3 & 3 & 3 & 3 & 3 & 3 & 6 & 6 & 6 & 6 & 6 & 6 & 6 & 6 & 6 & 6 & 6 & 6 & 0 & 0 & 0 & 0 & 0 & 0 & 0 & 0 & 0 & 0 & 0 & 0 \\
9 & 3 & 3 & 3 & 9 & 9 & 9 & 3 & 3 & 3 & 3 & 3 & 3 & 3 & 3 & 3 & 6 & 6 & 6 & 6 & 6 & 6 & 6 & 6 & 6 & 6 & 6 & 6 & 0 & 0 & 0 & 0 & 0 & 0 & 0 & 0 & 0 & 0 & 0 & 0 \\
9 & 3 & 3 & 3 & 9 & 9 & 9 & 3 & 3 & 3 & 3 & 3 & 3 & 3 & 3 & 3 & 6 & 6 & 6 & 6 & 6 & 6 & 6 & 6 & 6 & 6 & 6 & 6 & 0 & 0 & 0 & 0 & 0 & 0 & 0 & 0 & 0 & 0 & 0 & 0 \\
9 & 3 & 3 & 3 & 9 & 9 & 9 & 3 & 3 & 3 & 3 & 3 & 3 & 3 & 3 & 3 & 6 & 6 & 6 & 6 & 6 & 6 & 6 & 6 & 6 & 6 & 6 & 6 & 0 & 0 & 0 & 0 & 0 & 0 & 0 & 0 & 0 & 0 & 0 & 0 \\
9 & 3 & 3 & 3 & 9 & 9 & 9 & 3 & 3 & 3 & 3 & 3 & 3 & 3 & 3 & 3 & 6 & 6 & 6 & 6 & 6 & 6 & 6 & 6 & 6 & 6 & 6 & 6 & 0 & 0 & 0 & 0 & 0 & 0 & 0 & 0 & 0 & 0 & 0 & 0 \\
9 & 3 & 3 & 3 & 9 & 9 & 9 & 3 & 3 & 3 & 3 & 3 & 3 & 3 & 3 & 3 & 6 & 6 & 6 & 6 & 6 & 6 & 6 & 6 & 6 & 6 & 6 & 6 & 0 & 0 & 0 & 0 & 0 & 0 & 0 & 0 & 0 & 0 & 0 & 0 \\
9 & 3 & 3 & 3 & 9 & 9 & 9 & 3 & 3 & 3 & 3 & 3 & 3 & 3 & 3 & 3 & 6 & 6 & 6 & 6 & 6 & 6 & 6 & 6 & 6 & 6 & 6 & 6 & 0 & 0 & 0 & 0 & 0 & 0 & 0 & 0 & 0 & 0 & 0 & 0 \\
9 & 3 & 3 & 3 & 9 & 9 & 9 & 3 & 3 & 3 & 3 & 3 & 3 & 3 & 3 & 3 & 6 & 6 & 6 & 6 & 6 & 6 & 6 & 6 & 6 & 6 & 6 & 6 & 0 & 0 & 0 & 0 & 0 & 0 & 0 & 0 & 0 & 0 & 0 & 0 \\
9 & 3 & 3 & 3 & 9 & 9 & 9 & 3 & 3 & 3 & 3 & 3 & 3 & 3 & 3 & 3 & 6 & 6 & 6 & 6 & 6 & 6 & 6 & 6 & 6 & 6 & 6 & 6 & 0 & 0 & 0 & 0 & 0 & 0 & 0 & 0 & 0 & 0 & 0 & 0 \\
9 & 9 & 9 & 9 & 3 & 3 & 3 & 3 & 3 & 3 & 3 & 3 & 3 & 3 & 3 & 3 & 0 & 0 & 0 & 0 & 0 & 0 & 0 & 0 & 0 & 0 & 0 & 0 & 6 & 6 & 6 & 6 & 6 & 6 & 6 & 6 & 6 & 6 & 6 & 6 \\
9 & 9 & 9 & 9 & 3 & 3 & 3 & 3 & 3 & 3 & 3 & 3 & 3 & 3 & 3 & 3 & 0 & 0 & 0 & 0 & 0 & 0 & 0 & 0 & 0 & 0 & 0 & 0 & 6 & 6 & 6 & 6 & 6 & 6 & 6 & 6 & 6 & 6 & 6 & 6 \\
9 & 9 & 9 & 9 & 3 & 3 & 3 & 3 & 3 & 3 & 3 & 3 & 3 & 3 & 3 & 3 & 0 & 0 & 0 & 0 & 0 & 0 & 0 & 0 & 0 & 0 & 0 & 0 & 6 & 6 & 6 & 6 & 6 & 6 & 6 & 6 & 6 & 6 & 6 & 6 \\
9 & 9 & 9 & 9 & 3 & 3 & 3 & 3 & 3 & 3 & 3 & 3 & 3 & 3 & 3 & 3 & 0 & 0 & 0 & 0 & 0 & 0 & 0 & 0 & 0 & 0 & 0 & 0 & 6 & 6 & 6 & 6 & 6 & 6 & 6 & 6 & 6 & 6 & 6 & 6 \\
9 & 9 & 9 & 9 & 3 & 3 & 3 & 3 & 3 & 3 & 3 & 3 & 3 & 3 & 3 & 3 & 0 & 0 & 0 & 0 & 0 & 0 & 0 & 0 & 0 & 0 & 0 & 0 & 6 & 6 & 6 & 6 & 6 & 6 & 6 & 6 & 6 & 6 & 6 & 6 \\
9 & 9 & 9 & 9 & 3 & 3 & 3 & 3 & 3 & 3 & 3 & 3 & 3 & 3 & 3 & 3 & 0 & 0 & 0 & 0 & 0 & 0 & 0 & 0 & 0 & 0 & 0 & 0 & 6 & 6 & 6 & 6 & 6 & 6 & 6 & 6 & 6 & 6 & 6 & 6 \\
9 & 9 & 9 & 9 & 3 & 3 & 3 & 3 & 3 & 3 & 3 & 3 & 3 & 3 & 3 & 3 & 0 & 0 & 0 & 0 & 0 & 0 & 0 & 0 & 0 & 0 & 0 & 0 & 6 & 6 & 6 & 6 & 6 & 6 & 6 & 6 & 6 & 6 & 6 & 6 \\
9 & 9 & 9 & 9 & 3 & 3 & 3 & 3 & 3 & 3 & 3 & 3 & 3 & 3 & 3 & 3 & 0 & 0 & 0 & 0 & 0 & 0 & 0 & 0 & 0 & 0 & 0 & 0 & 6 & 6 & 6 & 6 & 6 & 6 & 6 & 6 & 6 & 6 & 6 & 6 \\
9 & 9 & 9 & 9 & 3 & 3 & 3 & 3 & 3 & 3 & 3 & 3 & 3 & 3 & 3 & 3 & 0 & 0 & 0 & 0 & 0 & 0 & 0 & 0 & 0 & 0 & 0 & 0 & 6 & 6 & 6 & 6 & 6 & 6 & 6 & 6 & 6 & 6 & 6 & 6 \\
9 & 9 & 9 & 9 & 3 & 3 & 3 & 3 & 3 & 3 & 3 & 3 & 3 & 3 & 3 & 3 & 0 & 0 & 0 & 0 & 0 & 0 & 0 & 0 & 0 & 0 & 0 & 0 & 6 & 6 & 6 & 6 & 6 & 6 & 6 & 6 & 6 & 6 & 6 & 6 \\
9 & 9 & 9 & 9 & 3 & 3 & 3 & 3 & 3 & 3 & 3 & 3 & 3 & 3 & 3 & 3 & 0 & 0 & 0 & 0 & 0 & 0 & 0 & 0 & 0 & 0 & 0 & 0 & 6 & 6 & 6 & 6 & 6 & 6 & 6 & 6 & 6 & 6 & 6 & 6 \\
9 & 9 & 9 & 9 & 3 & 3 & 3 & 3 & 3 & 3 & 3 & 3 & 3 & 3 & 3 & 3 & 0 & 0 & 0 & 0 & 0 & 0 & 0 & 0 & 0 & 0 & 0 & 0 & 6 & 6 & 6 & 6 & 6 & 6 & 6 & 6 & 6 & 6 & 6 & 6
\end{array}\right]$$}

\section{Proof of relevance to \eqref{equation:sos}}\label{section:relevance}

Starting with the notation and setup of Theorem~\ref{theorem:main}, the coefficient of $t^2$ in $\mathsf{trace}((A+tB)^6)$ is a polynomial $p$ in the variables $a_{\{i,j\}}$ and $b_{\{i,j\}}$. Note that $p$ is a multivariate polynomial where each term has total degree $6$. Moreover, each term as $a$-degree $4$ and $b$-degree $2$. Since $m=6$, if $n \geq 6$, the number of monomials up to type (under $\mathfrak{S}_n$-action $a_{\{i,j\}} \mapsto a_{\{\sigma(i),\sigma(j)\}}$ and $b_{\{i,j\}} \mapsto b_{\{\sigma(i),\sigma(j)\}}$) remains constant as $n \rightarrow \infty$.

For each $n$, the monomial types may be obtained by exhaustion, examining the terms of $p$. For example, when $n=3$, there are exactly $80$ monomial types. When $n=6$, there are exactly $169$ monomial types. Therefore, for all $n \geq 6$, the number of monomial types is $169$. Moreover, for all $n$, every term is equivalent (up to $\mathfrak{S}_n$-action) to one of these $169$ monomials: 
$ a_{11}^{4} b_{11}^{2} $,
$ a_{11}^{2} a_{12}^{2} b_{11}^{2} $,
$ a_{12}^{4} b_{11}^{2} $,
$ a_{12}^{2} a_{13}^{2} b_{11}^{2} $,
$ a_{11} a_{12}^{2} a_{22} b_{11}^{2} $,
$ a_{12}^{2} a_{22}^{2} b_{11}^{2} $,
$ a_{11} a_{12} a_{13} a_{23} b_{11}^{2} $,
$ a_{12} a_{13} a_{22} a_{23} b_{11}^{2} $,
$ a_{12}^{2} a_{23}^{2} b_{11}^{2} $,
$ a_{13} a_{14} a_{23} a_{24} b_{11}^{2} $,
$ a_{11}^{3} a_{12} b_{11} b_{12} $,
$ a_{11} a_{12}^{3} b_{11} b_{12} $,
$ a_{11} a_{12} a_{13}^{2} b_{11} b_{12} $,
$ a_{11}^{2} a_{12} a_{22} b_{11} b_{12} $,
$ a_{12}^{3} a_{22} b_{11} b_{12} $,
$ a_{12} a_{13}^{2} a_{22} b_{11} b_{12} $,
$ a_{11} a_{12} a_{22}^{2} b_{11} b_{12} $,
$ a_{12} a_{22}^{3} b_{11} b_{12} $,
$ a_{11}^{2} a_{13} a_{23} b_{11} b_{12} $,
$ a_{12}^{2} a_{13} a_{23} b_{11} b_{12} $,
$ a_{13}^{3} a_{23} b_{11} b_{12} $,
$ a_{13} a_{14}^{2} a_{23} b_{11} b_{12} $,
$ a_{11} a_{13} a_{22} a_{23} b_{11} b_{12} $,
$ a_{13} a_{22}^{2} a_{23} b_{11} b_{12} $,
$ a_{11} a_{12} a_{23}^{2} b_{11} b_{12} $,
$ a_{12} a_{22} a_{23}^{2} b_{11} b_{12} $,
$ a_{13} a_{23}^{3} b_{11} b_{12} $,
$ a_{14} a_{23}^{2} a_{24} b_{11} b_{12} $,
$ a_{12} a_{13}^{2} a_{33} b_{11} b_{12} $,
$ a_{11} a_{13} a_{23} a_{33} b_{11} b_{12} $,
$ a_{13} a_{22} a_{23} a_{33} b_{11} b_{12} $,
$ a_{12} a_{23}^{2} a_{33} b_{11} b_{12} $,
$ a_{13} a_{23} a_{33}^{2} b_{11} b_{12} $,
$ a_{12} a_{13} a_{14} a_{34} b_{11} b_{12} $,
$ a_{11} a_{14} a_{23} a_{34} b_{11} b_{12} $,
$ a_{14} a_{22} a_{23} a_{34} b_{11} b_{12} $,
$ a_{12} a_{23} a_{24} a_{34} b_{11} b_{12} $,
$ a_{14} a_{23} a_{33} a_{34} b_{11} b_{12} $,
$ a_{13} a_{24} a_{33} a_{34} b_{11} b_{12} $,
$ a_{13} a_{23} a_{34}^{2} b_{11} b_{12} $,
$ a_{15} a_{24} a_{34} a_{35} b_{11} b_{12} $,
$ a_{11}^{4} b_{12}^{2} $,
$ a_{11}^{2} a_{12}^{2} b_{12}^{2} $,
$ a_{12}^{4} b_{12}^{2} $,
$ a_{11}^{2} a_{13}^{2} b_{12}^{2} $,
$ a_{12}^{2} a_{13}^{2} b_{12}^{2} $,
$ a_{13}^{4} b_{12}^{2} $,
$ a_{13}^{2} a_{14}^{2} b_{12}^{2} $,
$ a_{11}^{3} a_{22} b_{12}^{2} $,
$ a_{11} a_{12}^{2} a_{22} b_{12}^{2} $,
$ a_{11} a_{13}^{2} a_{22} b_{12}^{2} $,
$ a_{11}^{2} a_{22}^{2} b_{12}^{2} $,
$ a_{13}^{2} a_{22}^{2} b_{12}^{2} $,
$ a_{11} a_{12} a_{13} a_{23} b_{12}^{2} $,
$ a_{13}^{2} a_{23}^{2} b_{12}^{2} $,
$ a_{14}^{2} a_{23}^{2} b_{12}^{2} $,
$ a_{13} a_{14} a_{23} a_{24} b_{12}^{2} $,
$ a_{11} a_{13}^{2} a_{33} b_{12}^{2} $,
$ a_{13}^{2} a_{22} a_{33} b_{12}^{2} $,
$ a_{12} a_{13} a_{23} a_{33} b_{12}^{2} $,
$ a_{13}^{2} a_{33}^{2} b_{12}^{2} $,
$ a_{11} a_{13} a_{14} a_{34} b_{12}^{2} $,
$ a_{13} a_{14} a_{22} a_{34} b_{12}^{2} $,
$ a_{12} a_{14} a_{23} a_{34} b_{12}^{2} $,
$ a_{13} a_{14} a_{33} a_{34} b_{12}^{2} $,
$ a_{13}^{2} a_{34}^{2} b_{12}^{2} $,
$ a_{14} a_{15} a_{34} a_{35} b_{12}^{2} $,
$ a_{11}^{2} a_{12} a_{13} b_{12} b_{13} $,
$ a_{12}^{3} a_{13} b_{12} b_{13} $,
$ a_{12} a_{13} a_{14}^{2} b_{12} b_{13} $,
$ a_{11} a_{12} a_{13} a_{22} b_{12} b_{13} $,
$ a_{12} a_{13} a_{22}^{2} b_{12} b_{13} $,
$ a_{11}^{3} a_{23} b_{12} b_{13} $,
$ a_{11} a_{12}^{2} a_{23} b_{12} b_{13} $,
$ a_{11} a_{14}^{2} a_{23} b_{12} b_{13} $,
$ a_{11}^{2} a_{22} a_{23} b_{12} b_{13} $,
$ a_{12}^{2} a_{22} a_{23} b_{12} b_{13} $,
$ a_{13}^{2} a_{22} a_{23} b_{12} b_{13} $,
$ a_{14}^{2} a_{22} a_{23} b_{12} b_{13} $,
$ a_{11} a_{22}^{2} a_{23} b_{12} b_{13} $,
$ a_{22}^{3} a_{23} b_{12} b_{13} $,
$ a_{12} a_{13} a_{23}^{2} b_{12} b_{13} $,
$ a_{11} a_{23}^{3} b_{12} b_{13} $,
$ a_{22} a_{23}^{3} b_{12} b_{13} $,
$ a_{11} a_{13} a_{14} a_{24} b_{12} b_{13} $,
$ a_{13} a_{14} a_{22} a_{24} b_{12} b_{13} $,
$ a_{12} a_{14} a_{23} a_{24} b_{12} b_{13} $,
$ a_{12} a_{13} a_{24}^{2} b_{12} b_{13} $,
$ a_{11} a_{23} a_{24}^{2} b_{12} b_{13} $,
$ a_{22} a_{23} a_{24}^{2} b_{12} b_{13} $,
$ a_{12} a_{13} a_{22} a_{33} b_{12} b_{13} $,
$ a_{11} a_{22} a_{23} a_{33} b_{12} b_{13} $,
$ a_{22}^{2} a_{23} a_{33} b_{12} b_{13} $,
$ a_{13} a_{14} a_{24} a_{33} b_{12} b_{13} $,
$ a_{23} a_{24}^{2} a_{33} b_{12} b_{13} $,
$ a_{11}^{2} a_{24} a_{34} b_{12} b_{13} $,
$ a_{12}^{2} a_{24} a_{34} b_{12} b_{13} $,
$ a_{14}^{2} a_{24} a_{34} b_{12} b_{13} $,
$ a_{15}^{2} a_{24} a_{34} b_{12} b_{13} $,
$ a_{11} a_{22} a_{24} a_{34} b_{12} b_{13} $,
$ a_{22}^{2} a_{24} a_{34} b_{12} b_{13} $,
$ a_{23}^{2} a_{24} a_{34} b_{12} b_{13} $,
$ a_{24}^{3} a_{34} b_{12} b_{13} $,
$ a_{14} a_{15} a_{25} a_{34} b_{12} b_{13} $,
$ a_{24} a_{25}^{2} a_{34} b_{12} b_{13} $,
$ a_{22} a_{24} a_{33} a_{34} b_{12} b_{13} $,
$ a_{14}^{2} a_{23} a_{44} b_{12} b_{13} $,
$ a_{13} a_{14} a_{24} a_{44} b_{12} b_{13} $,
$ a_{23} a_{24}^{2} a_{44} b_{12} b_{13} $,
$ a_{11} a_{24} a_{34} a_{44} b_{12} b_{13} $,
$ a_{22} a_{24} a_{34} a_{44} b_{12} b_{13} $,
$ a_{24} a_{34} a_{44}^{2} b_{12} b_{13} $,
$ a_{14} a_{15} a_{23} a_{45} b_{12} b_{13} $,
$ a_{13} a_{15} a_{24} a_{45} b_{12} b_{13} $,
$ a_{23} a_{24} a_{25} a_{45} b_{12} b_{13} $,
$ a_{11} a_{25} a_{34} a_{45} b_{12} b_{13} $,
$ a_{22} a_{25} a_{34} a_{45} b_{12} b_{13} $,
$ a_{25} a_{34} a_{44} a_{45} b_{12} b_{13} $,
$ a_{24} a_{34} a_{45}^{2} b_{12} b_{13} $,
$ a_{26} a_{35} a_{45} a_{46} b_{12} b_{13} $,
$ a_{11}^{2} a_{12}^{2} b_{11} b_{22} $,
$ a_{12}^{4} b_{11} b_{22} $,
$ a_{12}^{2} a_{13}^{2} b_{11} b_{22} $,
$ a_{11} a_{12}^{2} a_{22} b_{11} b_{22} $,
$ a_{11} a_{12} a_{13} a_{23} b_{11} b_{22} $,
$ a_{13}^{2} a_{23}^{2} b_{11} b_{22} $,
$ a_{13} a_{14} a_{23} a_{24} b_{11} b_{22} $,
$ a_{12} a_{13} a_{23} a_{33} b_{11} b_{22} $,
$ a_{12} a_{14} a_{23} a_{34} b_{11} b_{22} $,
$ a_{11} a_{12}^{2} a_{13} b_{13} b_{22} $,
$ a_{12}^{2} a_{13} a_{22} b_{13} b_{22} $,
$ a_{11}^{2} a_{12} a_{23} b_{13} b_{22} $,
$ a_{12}^{3} a_{23} b_{13} b_{22} $,
$ a_{12} a_{13}^{2} a_{23} b_{13} b_{22} $,
$ a_{12} a_{14}^{2} a_{23} b_{13} b_{22} $,
$ a_{11} a_{12} a_{22} a_{23} b_{13} b_{22} $,
$ a_{12} a_{22}^{2} a_{23} b_{13} b_{22} $,
$ a_{11} a_{13} a_{23}^{2} b_{13} b_{22} $,
$ a_{12} a_{13} a_{14} a_{24} b_{13} b_{22} $,
$ a_{11} a_{14} a_{23} a_{24} b_{13} b_{22} $,
$ a_{14} a_{22} a_{23} a_{24} b_{13} b_{22} $,
$ a_{12} a_{23} a_{24}^{2} b_{13} b_{22} $,
$ a_{11} a_{12} a_{23} a_{33} b_{13} b_{22} $,
$ a_{14} a_{23} a_{24} a_{33} b_{13} b_{22} $,
$ a_{12}^{2} a_{14} a_{34} b_{13} b_{22} $,
$ a_{14} a_{24}^{2} a_{34} b_{13} b_{22} $,
$ a_{15} a_{24} a_{25} a_{34} b_{13} b_{22} $,
$ a_{14} a_{23} a_{24} a_{44} b_{13} b_{22} $,
$ a_{15} a_{23} a_{24} a_{45} b_{13} b_{22} $,
$ a_{11} a_{12} a_{13} a_{14} b_{14} b_{23} $,
$ a_{12} a_{13} a_{14} a_{22} b_{14} b_{23} $,
$ a_{12}^{2} a_{14} a_{23} b_{14} b_{23} $,
$ a_{11}^{2} a_{13} a_{24} b_{14} b_{23} $,
$ a_{12}^{2} a_{13} a_{24} b_{14} b_{23} $,
$ a_{13}^{3} a_{24} b_{14} b_{23} $,
$ a_{13} a_{14}^{2} a_{24} b_{14} b_{23} $,
$ a_{13} a_{15}^{2} a_{24} b_{14} b_{23} $,
$ a_{11} a_{13} a_{22} a_{24} b_{14} b_{23} $,
$ a_{13} a_{14} a_{15} a_{25} b_{14} b_{23} $,
$ a_{11} a_{13} a_{24} a_{33} b_{14} b_{23} $,
$ a_{13} a_{22} a_{24} a_{33} b_{14} b_{23} $,
$ a_{12} a_{23} a_{24} a_{33} b_{14} b_{23} $,
$ a_{11} a_{15} a_{25} a_{34} b_{14} b_{23} $,
$ a_{15} a_{25} a_{33} a_{34} b_{14} b_{23} $,
$ a_{12} a_{24} a_{25} a_{35} b_{14} b_{23} $,
$ a_{15} a_{25} a_{35} a_{45} b_{14} b_{23} $,
$ a_{16} a_{26} a_{35} a_{45} b_{14} b_{23} $,
$ a_{15} a_{25} a_{34} a_{55} b_{14} b_{23} $,
$ a_{16} a_{25} a_{34} a_{56} b_{14} b_{23} $.

For each monomial type, every monomial of the given type may be considered. (This can be completed by exhaustion, applying every element of $\mathfrak{S}_n$ to the fixed monomial type.)

For each monomial, we state which entries of $U$ and which entries (in which copy) of $R$ addresses said monomial. Mentioning every entry in full this way is exhaustive, but by examining the mentioned entries by their indexes, we have a finite-length (albeit long) description of showing that each matrix entry is used exactly once. The verification in the case of $n=6$, when replacing constant indices with variables, is a computer-aided verification (in $169$ cases) that for all $n$, the coefficient of $t^2$ in $\mathsf{trace}((A+tB)^6)$ is \eqref{equation:sos}, as desired. Details of this verification (for $n=3,4,5,6$) and accompanying sage code (see~\cite{sage}) are found at
\begin{center}
\url{https://edward-kim-math.github.io/cdta-m6r2}
\end{center}
but for the sake of illustration, in the case when $n=3$, the following three monomials are the only monomials of $p$ appearing together in a type class (up to $\mathfrak{S}_3$-action), namely, $ a_{12}^{2} a_{13}^{2} b_{11}^{2} $ and $ a_{12}^{2} a_{23}^{2} b_{22}^{2} $ and $ a_{13}^{2} a_{23}^{2} b_{33}^{2} $. Each of these three monomials appears in $p$ with coefficient $18$. The list below describes which entries of $U$ and which entries in which copy of $R$ address each monomial:
\begin{itemize}

\item Monomial $ a_{12}^{2} a_{13}^{2} b_{11}^{2} $
\begin{itemize}
\item The $(\hat{1}, \hat{2})$-row $(\hat{1}, \hat{3})$-column entry of $U$ is $6$, addresses $( a_{12}^{2} b_{11} )( a_{13}^{2} b_{11} )$ 
\item The $(\hat{1}, \hat{3})$-row $(\hat{1}, \hat{2})$-column entry of $U$ is $6$, addresses $( a_{13}^{2} b_{11} )( a_{12}^{2} b_{11} )$ 
\item The $ \left(2, 3\right) $-copy of $R$ whose entry  $ 6 $ is addressed by $( a_{12} a_{13} b_{11} )( a_{12} a_{13} b_{11} )$ 
\end{itemize}

\item Monomial $ a_{12}^{2} a_{23}^{2} b_{22}^{2} $
\begin{itemize}
\item The $(\hat{2}, \hat{1})$-row $(\hat{2}, \hat{3})$-column entry of $U$ is $6$, addresses $( a_{12}^{2} b_{22} )( a_{23}^{2} b_{22} )$ 
\item The $(\hat{2}, \hat{3})$-row $(\hat{2}, \hat{1})$-column entry of $U$ is $6$, addresses $( a_{23}^{2} b_{22} )( a_{12}^{2} b_{22} )$ 
\item The $ \left(1, 3\right) $-copy of $R$ whose entry  $ 6 $ is addressed by $( a_{12} a_{23} b_{22} )( a_{12} a_{23} b_{22} )$ 
\end{itemize}

\item Monomial $ a_{13}^{2} a_{23}^{2} b_{33}^{2} $
\begin{itemize}
\item The $(\hat{3}, \hat{1})$-row $(\hat{3}, \hat{2})$-column entry of $U$ is $6$, addresses $( a_{13}^{2} b_{33} )( a_{23}^{2} b_{33} )$ 
\item The $(\hat{3}, \hat{2})$-row $(\hat{3}, \hat{1})$-column entry of $U$ is $6$, addresses $( a_{23}^{2} b_{33} )( a_{13}^{2} b_{33} )$ 
\item The $ \left(1, 2\right) $-copy of $R$ whose entry  $ 6 $ is addressed by $( a_{13} a_{23} b_{33} )( a_{13} a_{23} b_{33} )$ 
\end{itemize}

\end{itemize}

\section{Proof of positive semidefiniteness of $R$}\label{section:psdRn}

Fix an integer $n \geq 4$. Our proof that $R$ is positive semidefinite makes use of the generalized Schur complement, which relies on the Moore-Penrose pseudo-inverse (see~\cite{BenIsraelThomas, Moore, Penrose}). For a matrix $C$ with real entries, its pseudo-inverse $C^+$ satisfies $CC^+C=C$ and $C^+CC^+=C^+$ and $(CC^+)^t=CC^+$ and $(C^+C)^t=C^+C$. Given the matrix
\[ R=\left[\begin{array}{cc}F&G^t \\ G & E\end{array}\right],\]
the generalized Schur complement (see~\cite{Zhang}) of $E$ in $R$ is $H-G^tE^+G$.

\begin{lemma}
Suppose $C$ is an $s \times s$ matrix with every entry $c$. Then the pseudo-inverse of $C$ is an $s \times s$ matrix with every entry $d = \frac{1}{s^2c}$.
\end{lemma}
\begin{proof}
Let $C$ be an $s \times s$ matrix with every entry being $c$. Let $D$ be an $s \times s$ matrix with every entry $d$. Then $(CD)^t=CD$ and $(DC)^t=DC$ are obvious since $C$ and $D$ are themselves symmetric.

We need $CDC=C$. Note $CD$ has every entry $scd$. So $(CD)C$ is a matrix where each entry is $s(scd)c$ and for this to equal $c$, we need $d = \frac{1}{s^2c}$. Similarly, for $DCD=D$ to be true, we again need $d = \frac{1}{s^2c}$.
\end{proof}
The submatrix of $R$ consisting of all rows and columns from block $5$ is the $n(n-1) \times n(n-1)$ with every entry $6$, so its pseudo-inverse is the $n(n-1) \times n(n-1)$ matrix with every entry $\frac{1}{6n^2(n-1)^2}$. Let $E$ be the submatrix of $R$ consisting of all rows and columns from blocks $5$ and $6$. That is, $E$ is the $2n(n-1) \times 2n(n-1)$ matrix
%\[ E = \left[\begin{array}{cc}6J & 0J \\ 0J & 6J\end{array}\right],\]
%where $J$ is the $n(n-1) \times n(n-1)$ matrix with every entry $1$.
\[
\begin{bNiceArray}{cc}[first-row,first-col]
& n(n-1) & n(n-1) \\
n(n-1) & 6 & 0 \\
n(n-1) & 0 & 6
\end{bNiceArray}
.\]
using the notation introduced in Section~\ref{section:block-constant-matrices} for block constant matrices.

\begin{corollary}\label{corollary:E-pseudoinverse-entries}
The pseudo-inverse $E^+$ of $E$ is
% \[E^+ = \left[\begin{array}{cc}
% \frac{1}{6n^2(n-1)^2}J & 0J\\
% 0J& \frac{1}{6n^2(n-1)^2}J
% \end{array}\right].\]
\[
\begin{bNiceArray}{cc}[first-row,first-col]
& n(n-1) & n(n-1) \\
n(n-1) & \frac{1}{6n^2(n-1)^2} & 0 \\
n(n-1) & 0 & \frac{1}{6n^2(n-1)^2}
\end{bNiceArray}
.\]
\end{corollary}
In what follows, we make regular use of the fact that a square matrix with all entries $c \geq 0$ is positive semidefinite since such a matrix is clearly a Gram matrix.

Now we make use of the generalized Schur complement. We decompose the matrix $R$ as
\begin{equation}\label{equation:block-decomposition-for-GSC}
R = \left[\begin{array}{cc}F&G^t \\ G & E\end{array}\right]
\end{equation}
where $E$ has all rows and columns from blocks $5$ and $6$, while $F$ has all rows and columns from blocks $1$, $2$, $3$, and $4$, with $G$ having the appropriate rows and columns. It is known (see~\cite{Zhang}) that a matrix $R$ with the block decomposition as shown in \eqref{equation:block-decomposition-for-GSC} is positive semidefinite if and only if $E$ is positive semidefinite and $(I-EE^T)G = 0$ and $F-G^tE^+G$ is positive semidefinite. 

\begin{lemma}
$E$ is positive semidefinite.
\end{lemma}
\begin{proof}
Note $E$ is a block diagonal matrix, and $E$ is $6$ as each entry in the two diagonal blocks, so each block is a Gram matrix.
\end{proof}

\begin{lemma}
Let $I$ be the $2n(n-1) \times 2n(n-1)$ identity matrix, and let $E$ and $G$ be as defined before. Then $(I-EE^+)G$ is the $2n(n-1) \times [1+2(n-1) + (n-1)^2]$ matrix of all zeroes.
\end{lemma}
\begin{proof}
It follows from Corollary~\ref{corollary:E-pseudoinverse-entries} that $E E^+$ is a $2n(n-1) \times 2n(n-1)$ matrix whose upper left block and lower right blocks are $n(n-1) \times n(n-1)$ matrices with every entry $n(n-1) \cdot 6 \cdot \frac{1}{6n^2(n-1)^2} = \frac{1}{n(n-1)}$, and whose upper right block and lower left block are all zeroes. Thus, the $(i,j)$-entry of $H = I - E E^{+}$ is
\[
H_{i,j} = \left\{
\begin{array}{ll}
1 - \frac{1}{n(n-1)} & \text{if } i = j \\
- \frac{1}{n(n-1)} & \text{if } i \not= j \text{ and } (i,j \leq n \text{ or } i,j > n) \\
0 & \text{otherwise.}
\end{array}
\right.
\]
Every column of $G$ is a vector with $2n(n-1)$ entries, where the first $n(n-1)$ entries are identical (either $3$ or $9$), and the second $n(n-1)$ entries are identical (either $3$ or $9$). Every one of these combinations occurs: that is, there is a column of all $9$s, but there is also a column of $G$ where the first half is all $3$s and the second half is all $9$s, and so on.

Now we prove $HG=0$. For this, consider an arbitrary row of $H$ and an arbitrary column of $G$, and we will consider the result of the dot product. For the chosen column of $G$, let us say that the first $n(n-1)$ entries are $\rho$ and the second $n(n-1)$ entries are $\tau$. For the selected row of $H$, either the $0$s are the second-half entries or the first-half entries. If the second-half entries are $0$, then the product is
\[ \textstyle (1 - \frac{1}{n(n-1)}) \rho + (n(n-1)-1) (\frac{1}{n(n-1)}) \rho + n(n-1) \cdot 0 \cdot \tau = n(n-1) \cdot \frac1{n(n-1)} \cdot \rho - \rho = 0. \]
If the chosen row of $H$ is in the second half (so the $0$ entries are in the first half, then the expression is similar to the one just shown, except that the $\rho$s and the $\tau$s are switched. In particular, the result is still $0$, so every entry of $HG$ is $0$.
\end{proof}

\begin{lemma}
The generalized Schur complement $F-G^tE^+G$ is positive semidefinite.
\end{lemma}
\begin{proof}
The indexing for $E^+$ follows the indexing for $E$. Thus, $E^+G$ has rows from blocks $5$ and $6$, and columns from blocks $1,2,3,4$. For each entry from row block $5$ and column block $1$, the entry of $E^+G$ is $n(n-1) \cdot \frac1{6n^2(n-1)^2} \cdot 9 + n(n-1) \cdot 0 \cdot 9 = \frac{9}{6n(n-1)}$. Each entry of $E^+G$ is similarly computed. Thus, $E^+G$ is the product of the following two block constant matrices
\[
\begin{bNiceArray}{cc}[first-row,first-col]
& n(n-1) & n(n-1) \\
n(n-1) & \frac{1}{6n^2(n-1)^2} & 0 \\
n(n-1) & 0 & \frac{1}{6n^2(n-1)^2}
\end{bNiceArray}
\,\,\,\,\,\,\,
\begin{bNiceArray}{cccc}[first-row,first-col]
& 1 & n-1 & n-1 & (n-1)^2 \\
n(n-1) & 9 & 3 & 9 & 3\\
n(n-1) & 9 & 9 & 3 & 3
\end{bNiceArray}
\]
which simplifies to
\[
\begin{bNiceArray}{cccc}[first-row,first-col]
& 1 & n-1 & n-1 & (n-1)^2 \\
n(n-1) & \frac{9}{6n(n-1)} & \frac{3}{6n(n-1)} & \frac{9}{6n(n-1)} & \frac{3}{6n(n-1)}\\
n(n-1) & \frac{9}{6n(n-1)} & \frac{9}{6n(n-1)} & \frac{3}{6n(n-1)} & \frac{3}{6n(n-1)}
\end{bNiceArray}
.\]

We now examine $G^t(E^+G)$. For example, the (only) entry whose row block is $1$ and whose column block is $1$ has value $n(n-1) \cdot 9 \cdot \frac{9}{6n(n-1)} + n(n-1) \cdot 9 \cdot \frac{9}{6n(n-1)} = \frac{9 \cdot 9 + 9 \cdot 9}{6}= 27$. The remaining entries are found the same way. Thus, $G^t(E^+G)$ is the product of the following two block constant matrices
\[
\begin{bNiceArray}{cc}[first-row,first-col]
& n(n-1) & n(n-1) \\
1 & 9 & 9 \\
n-1 & 3 & 9 \\
n-1 & 9 & 3 \\
(n-1)^2 & 3 & 3
\end{bNiceArray}
\,\,\,\,\,\,\,
\begin{bNiceArray}{cccc}[first-row,first-col]
& 1 & n-1 & n-1 & (n-1)^2 \\
n(n-1) & \frac{9}{6n(n-1)} & \frac{3}{6n(n-1)} & \frac{9}{6n(n-1)} & \frac{3}{6n(n-1)}\\
n(n-1) & \frac{9}{6n(n-1)} & \frac{9}{6n(n-1)} & \frac{3}{6n(n-1)} & \frac{3}{6n(n-1)}
\end{bNiceArray}
\]
which simplifies to
\[
\begin{bNiceArray}{cccc}[first-row,first-col]
& 1 & n-1 & n-1 & (n-1)^2 \\
1       & \frac{9 \cdot 9 + 9 \cdot 9}{6} & \frac{9 \cdot 3 + 9 \cdot 9}{6} & \frac{9 \cdot 9 + 9 \cdot 3}{6} & \frac{9 \cdot 3 + 9 \cdot 3}{6} \\
n-1     & \frac{3 \cdot 9 + 9 \cdot 9}{6} & \frac{3 \cdot 3 + 9 \cdot 9}{6} & \frac{3 \cdot 9 + 9 \cdot 3}{6} & \frac{3 \cdot 3 + 9 \cdot 3}{6} \\
n-1     & \frac{9 \cdot 9 + 3 \cdot 9}{6} & \frac{9 \cdot 3 + 3 \cdot 9}{6} & \frac{9 \cdot 9 + 3 \cdot 3}{6} & \frac{9 \cdot 3 + 3 \cdot 3}{6} \\
(n-1)^2 & \frac{3 \cdot 9 + 3 \cdot 9}{6} & \frac{3 \cdot 3 + 3 \cdot 9}{6} & \frac{3 \cdot 9 + 3 \cdot 3}{6} & \frac{3 \cdot 3 + 3 \cdot 3}{6}
\end{bNiceArray}
=
\begin{bNiceArray}{cccc}[first-row,first-col]
& 1 & n-1 & n-1 & (n-1)^2 \\
1       & 27 & 18 & 18 & 9 \\
n-1     & 18 & 15 & 9 & 6 \\
n-1     & 18 & 9 & 15 & 6 \\
(n-1)^2 & 9 & 6 & 6 & 3
\end{bNiceArray}.
\]
Then $F- G^tE^+G$ is
\[
\begin{bNiceArray}{cccc}[first-row,first-col]
& 1 & n-1 & n-1 & (n-1)^2 \\
1 & 30& 21& 21& 12\\
n-1 & 21& 18& 12& 9\\
n-1 & 21& 12& 18& 9\\
(n-1)^2 & 12& 9& 9& 6\\
\end{bNiceArray}
\,\,\,\,\,\,\,
-
\,\,\,\,\,\,\,
\begin{bNiceArray}{cccc}[first-row,first-col]
& 1 & n-1 & n-1 & (n-1)^2 \\
1       & 27 & 18 & 18 & 9 \\
n-1     & 18 & 15 & 9 & 6 \\
n-1     & 18 & 9 & 15 & 6 \\
(n-1)^2 & 9 & 6 & 6 & 3
\end{bNiceArray}
\,\,\,\,\,\,\,
=
\,\,\,\,\,\,\,
\begin{bNiceArray}{c}[first-row,first-col]
& n^2 \\
n^2 & 3
\end{bNiceArray}
.
\]
Therefore $F-G^tE^+G$ is positive semidefinite since it is an $n^2 \times n^2$ matrix with every entry $3$.
\end{proof}

\section{Proof of positive semidefiniteness of $U$}\label{section:psdU} 

%%%% since the ({1,2,3},1) - ({4,5,6},4) entry is shown should this say fix an integer n \geq 6?

Fix an integer $n \geq 4$. In the following subsections, we describe the entries of $U$, $U^2$, $U^3$, and $U^4$ based on the indexing of rows and columns. In each of these four matrices, due to symmetry of an $\mathfrak{S}_n$ action, we only need to present one entry up to index type. Moreover, the size of this information does not grow as $n$ increases, and this is all that is needed to show that $U$ satisfies a certain polynomial whose degree is bounded universally in $n$.

\subsection{Entries of $U$}

We describe the matrix $U$ defined in Section~\ref{section:defU} by index type. While this expands on information which is essentially already provided by Section~\ref{section:defU}, formatting the information in this expanded form (and in this order, matching the next several subsections) will be useful in proving that $U$ satisfies a certain polynomial. In the list below, the indexing applies only when it makes sense. For example, discussing the $(\{4, 5, 6\}, 4)$-column does not apply when $n=4$.
\begin{itemize}
\item The $(\{1, 2, 3\}, 1)$-row $(\{1, 2, 3\}, 1)$-column entry is $24$.
\item The $(\{1, 2, 3\}, 1)$-row $(\{1, 2, 3\}, 2)$-column entry is $18$.
\item The $(\{1, 2, 3\}, 1)$-row $(\{1, 2, 4\}, 1)$-column entry is $18$.
\item The $(\{1, 2, 3\}, 1)$-row $(\{1, 2, 4\}, 2)$-column entry is $12$.
\item The $(\{1, 2, 3\}, 1)$-row $(\{1, 2, 4\}, 4)$-column entry is $12$.
\item The $(\{1, 2, 3\}, 3)$-row $(\{1, 2, 4\}, 4)$-column entry is $12$.
\item The $(\{1, 2, 3\}, 1)$-row $(\{1, 4, 5\}, 1)$-column entry is $12$.
\item The $(\{1, 2, 3\}, 1)$-row $(\{1, 4, 5\}, 4)$-column entry is $6$.
\item The $(\{1, 2, 3\}, 2)$-row $(\{1, 4, 5\}, 4)$-column entry is $6$.
\item The $(\{1, 2, 3\}, 1)$-row $(\{4, 5, 6\}, 4)$-column entry is $0$.
\item The $(1, 2)$-row $(1, 2)$-column entry is $36$.
\item The $(1, 2)$-row $(1, 3)$-column entry is $30$.
\item The $(1, 2)$-row $(2, 1)$-column entry is $24$.
\item The $(1, 2)$-row $(2, 3)$-column entry is $12$.
\item The $(1, 2)$-row $(3, 1)$-column entry is $12$.
\item The $(1, 2)$-row $(3, 2)$-column entry is $6$.
\item The $(1, 2)$-row $(3, 4)$-column entry is $0$.
\item The $(\hat{1}, \hat{2})$-row $(\hat{1}, \hat{1})$-column entry is $9$.
\item The $(\hat{1}, \hat{1})$-row $(\hat{1}, \hat{2})$-column entry is $9$.
\item The $(\hat{1}, \hat{1})$-row $(\hat{2}, \hat{2})$-column entry is $0$.
\item The $(\hat{1}, \hat{1})$-row $(\hat{2}, \hat{1})$-column entry is $6$.
\item The $(\hat{1}, \hat{1})$-row $(\hat{2}, \hat{3})$-column entry is $0$.
\item The $(\hat{1}, \hat{2})$-row $(\hat{1}, \hat{2})$-column entry is $9$.
\item The $(\hat{1}, \hat{2})$-row $(\hat{2}, \hat{1})$-column entry is $6$.
\item The $(\hat{1}, \hat{2})$-row $(\hat{1}, \hat{3})$-column entry is $6$.
\item The $(\hat{1}, \hat{2})$-row $(\hat{3}, \hat{1})$-column entry is $3$.
\item The $(\hat{1}, \hat{2})$-row $(\hat{2}, \hat{3})$-column entry is $3$.
\item The $(\hat{1}, \hat{2})$-row $(\hat{3}, \hat{2})$-column entry is $3$.
\item The $(\hat{1}, \hat{2})$-row $(\hat{3}, \hat{4})$-column entry is $0$.
\item The $(\{1, 2, 3\}, 1)$-row $(1, 2)$-column entry is $24$.
\item The $(\{1, 2, 3\}, 2)$-row $(1, 2)$-column entry is $18$.
\item The $(\{1, 2, 3\}, 3)$-row $(1, 2)$-column entry is $18$.
\item The $(\{1, 2, 3\}, 1)$-row $(1, 4)$-column entry is $18$.
\item The $(\{1, 2, 3\}, 2)$-row $(1, 4)$-column entry is $12$.
\item The $(\{1, 2, 3\}, 1)$-row $(4, 1)$-column entry is $6$.
\item The $(\{1, 2, 3\}, 2)$-row $(4, 1)$-column entry is $6$.
\item The $(\{1, 2, 3\}, 1)$-row $(4, 5)$-column entry is $0$.
\item The $(\{1, 2, 3\}, 1)$-row $(\hat{1}, \hat{1})$-column entry is $12$.
\item The $(\{1, 2, 3\}, 2)$-row $(\hat{1}, \hat{1})$-column entry is $9$.
\item The $(\{1, 2, 3\}, 1)$-row $(\hat{1}, \hat{2})$-column entry is $9$.
\item The $(\{1, 2, 3\}, 2)$-row $(\hat{1}, \hat{2})$-column entry is $12$.
\item The $(\{1, 2, 3\}, 3)$-row $(\hat{1}, \hat{2})$-column entry is $9$.
\item The $(\{1, 2, 3\}, 1)$-row $(\hat{1}, \hat{4})$-column entry is $6$.
\item The $(\{1, 2, 3\}, 2)$-row $(\hat{1}, \hat{4})$-column entry is $6$.
\item The $(\{1, 2, 3\}, 1)$-row $(\hat{4}, \hat{1})$-column entry is $6$.
\item The $(\{1, 2, 3\}, 2)$-row $(\hat{4}, \hat{1})$-column entry is $3$.
\item The $(\{1, 2, 3\}, 1)$-row $(\hat{4}, \hat{4})$-column entry is $0$.
\item The $(\{1, 2, 3\}, 1)$-row $(\hat{4}, \hat{5})$-column entry is $0$.
\item The $(1, 2)$-row $(\hat{1}, \hat{1})$-column entry is $21$.
\item The $(1, 2)$-row $(\hat{1}, \hat{2})$-column entry is $15$.
\item The $(1, 2)$-row $(\hat{1}, \hat{3})$-column entry is $12$.
\item The $(1, 2)$-row $(\hat{2}, \hat{1})$-column entry is $15$.
\item The $(1, 2)$-row $(\hat{2}, \hat{2})$-column entry is $9$.
\item The $(1, 2)$-row $(\hat{2}, \hat{3})$-column entry is $6$.
\item The $(1, 2)$-row $(\hat{3}, \hat{1})$-column entry is $9$.
\item The $(1, 2)$-row $(\hat{3}, \hat{2})$-column entry is $3$.
\item The $(1, 2)$-row $(\hat{3}, \hat{4})$-column entry is $0$.
\end{itemize}

\subsection{Entries of $U^2$}

This section describes the entries of $U^2$ up to index type. Detailed proofs for several entries are found at 
\begin{center}
\url{https://edward-kim-math.github.io/cdta-m6r2}
\end{center}
The entries of $U^2$ up to index type are:
\begin{itemize}
\item The $(\{1, 2, 3\}, 1)$-row $(\{1, 2, 3\}, 1)$-column entry is $216n^{2} + 1026n - 432$.
\item The $(\{1, 2, 3\}, 1)$-row $(\{1, 2, 3\}, 2)$-column entry is $198n^{2} + 1035n - 432$.
\item The $(\{1, 2, 3\}, 1)$-row $(\{1, 2, 4\}, 1)$-column entry is $162n^{2} + 1071n - 432$.
\item The $(\{1, 2, 3\}, 1)$-row $(\{1, 2, 4\}, 2)$-column entry is $144n^{2} + 1080n - 432$.
\item The $(\{1, 2, 3\}, 1)$-row $(\{1, 2, 4\}, 4)$-column entry is $126n^{2} + 1089n - 432$.
\item The $(\{1, 2, 3\}, 3)$-row $(\{1, 2, 4\}, 4)$-column entry is $108n^{2} + 1098n - 432$.
\item The $(\{1, 2, 3\}, 1)$-row $(\{1, 4, 5\}, 1)$-column entry is $108n^{2} + 1116n - 432$.
\item The $(\{1, 2, 3\}, 1)$-row $(\{1, 4, 5\}, 4)$-column entry is $72n^{2} + 1134n - 432$.
\item The $(\{1, 2, 3\}, 2)$-row $(\{1, 4, 5\}, 4)$-column entry is $54n^{2} + 1143n - 432$.
\item The $(\{1, 2, 3\}, 1)$-row $(\{4, 5, 6\}, 4)$-column entry is $1188n - 432$.
\item The $(1, 2)$-row $(1, 2)$-column entry is $360n^{2} + 918n - 432$.
\item The $(1, 2)$-row $(1, 3)$-column entry is $306n^{2} + 963n - 432$.
\item The $(1, 2)$-row $(2, 1)$-column entry is $252n^{2} + 990n - 432$.
\item The $(1, 2)$-row $(2, 3)$-column entry is $126n^{2} + 1089n - 432$.
\item The $(1, 2)$-row $(3, 1)$-column entry is $126n^{2} + 1089n - 432$.
\item The $(1, 2)$-row $(3, 2)$-column entry is $54n^{2} + 1143n - 432$.
\item The $(1, 2)$-row $(3, 4)$-column entry is $1188n - 432$.
\item The $(\hat{1}, \hat{2})$-row $(\hat{1}, \hat{1})$-column entry is $90n^{2} + 225n - 108$.
\item The $(\hat{1}, \hat{1})$-row $(\hat{1}, \hat{2})$-column entry is $90n^{2} + 225n - 108$.
\item The $(\hat{1}, \hat{1})$-row $(\hat{2}, \hat{2})$-column entry is $297n - 108$.
\item The $(\hat{1}, \hat{1})$-row $(\hat{2}, \hat{1})$-column entry is $63n^{2} + 252n - 108$.
\item The $(\hat{1}, \hat{1})$-row $(\hat{2}, \hat{3})$-column entry is $297n - 108$.
\item The $(\hat{1}, \hat{2})$-row $(\hat{1}, \hat{2})$-column entry is $81n^{2} + 234n - 108$.
\item The $(\hat{1}, \hat{2})$-row $(\hat{2}, \hat{1})$-column entry is $72n^{2} + 243n - 108$.
\item The $(\hat{1}, \hat{2})$-row $(\hat{1}, \hat{3})$-column entry is $54n^{2} + 252n - 108$.
\item The $(\hat{1}, \hat{2})$-row $(\hat{3}, \hat{1})$-column entry is $36n^{2} + 270n - 108$.
\item The $(\hat{1}, \hat{2})$-row $(\hat{2}, \hat{3})$-column entry is $36n^{2} + 270n - 108$.
\item The $(\hat{1}, \hat{2})$-row $(\hat{3}, \hat{2})$-column entry is $27n^{2} + 279n - 108$.
\item The $(\hat{1}, \hat{2})$-row $(\hat{3}, \hat{4})$-column entry is $297n - 108$.
\item The $(\{1, 2, 3\}, 1)$-row $(1, 2)$-column entry is $234n^{2} + 1017n - 432$.
\item The $(\{1, 2, 3\}, 2)$-row $(1, 2)$-column entry is $198n^{2} + 1035n - 432$.
\item The $(\{1, 2, 3\}, 3)$-row $(1, 2)$-column entry is $180n^{2} + 1044n - 432$.
\item The $(\{1, 2, 3\}, 1)$-row $(1, 4)$-column entry is $180n^{2} + 1062n - 432$.
\item The $(\{1, 2, 3\}, 2)$-row $(1, 4)$-column entry is $126n^{2} + 1089n - 432$.
\item The $(\{1, 2, 3\}, 1)$-row $(4, 1)$-column entry is $72n^{2} + 1134n - 432$.
\item The $(\{1, 2, 3\}, 2)$-row $(4, 1)$-column entry is $54n^{2} + 1143n - 432$.
\item The $(\{1, 2, 3\}, 1)$-row $(4, 5)$-column entry is $1188n - 432$.
\item The $(\{1, 2, 3\}, 1)$-row $(\hat{1}, \hat{1})$-column entry is $126n^{2} + 504n - 216$.
\item The $(\{1, 2, 3\}, 2)$-row $(\hat{1}, \hat{1})$-column entry is $90n^{2} + 522n - 216$.
\item The $(\{1, 2, 3\}, 1)$-row $(\hat{1}, \hat{2})$-column entry is $108n^{2} + 513n - 216$.
\item The $(\{1, 2, 3\}, 2)$-row $(\hat{1}, \hat{2})$-column entry is $108n^{2} + 513n - 216$.
\item The $(\{1, 2, 3\}, 3)$-row $(\hat{1}, \hat{2})$-column entry is $90n^{2} + 522n - 216$.
\item The $(\{1, 2, 3\}, 1)$-row $(\hat{1}, \hat{4})$-column entry is $72n^{2} + 540n - 216$.
\item The $(\{1, 2, 3\}, 2)$-row $(\hat{1}, \hat{4})$-column entry is $54n^{2} + 549n - 216$.
\item The $(\{1, 2, 3\}, 1)$-row $(\hat{4}, \hat{1})$-column entry is $54n^{2} + 558n - 216$.
\item The $(\{1, 2, 3\}, 2)$-row $(\hat{4}, \hat{1})$-column entry is $36n^{2} + 567n - 216$.
\item The $(\{1, 2, 3\}, 1)$-row $(\hat{4}, \hat{4})$-column entry is $594n - 216$.
\item The $(\{1, 2, 3\}, 1)$-row $(\hat{4}, \hat{5})$-column entry is $594n - 216$.
\item The $(1, 2)$-row $(\hat{1}, \hat{1})$-column entry is $216n^{2} + 432n - 216$.
\item The $(1, 2)$-row $(\hat{1}, \hat{2})$-column entry is $162n^{2} + 468n - 216$.
\item The $(1, 2)$-row $(\hat{1}, \hat{3})$-column entry is $126n^{2} + 495n - 216$.
\item The $(1, 2)$-row $(\hat{2}, \hat{1})$-column entry is $144n^{2} + 486n - 216$.
\item The $(1, 2)$-row $(\hat{2}, \hat{2})$-column entry is $90n^{2} + 522n - 216$.
\item The $(1, 2)$-row $(\hat{2}, \hat{3})$-column entry is $54n^{2} + 549n - 216$.
\item The $(1, 2)$-row $(\hat{3}, \hat{1})$-column entry is $90n^{2} + 531n - 216$.
\item The $(1, 2)$-row $(\hat{3}, \hat{2})$-column entry is $36n^{2} + 567n - 216$.
\item The $(1, 2)$-row $(\hat{3}, \hat{4})$-column entry is $594n - 216$.
\end{itemize}

\subsection{Entries of $U^3$}

The entries of $U^3$ can be obtained using identical methods used for the entries of $U^2$. The entries of $U^3$ up to index type are:
    \begin{itemize}
\item The $(\{1, 2, 3\}, 1)$-row $(\{1, 2, 3\}, 1)$-column entry is $2160n^{4} + 44496n^{3} - 43254n^{2} + 9828n$.
\item The $(\{1, 2, 3\}, 1)$-row $(\{1, 2, 3\}, 2)$-column entry is $2052n^{4} + 44604n^{3} - 43281n^{2} + 9828n$.
\item The $(\{1, 2, 3\}, 1)$-row $(\{1, 2, 4\}, 1)$-column entry is $1620n^{4} + 45360n^{3} - 43605n^{2} + 9828n$.
\item The $(\{1, 2, 3\}, 1)$-row $(\{1, 2, 4\}, 2)$-column entry is $1512n^{4} + 45468n^{3} - 43632n^{2} + 9828n$.
\item The $(\{1, 2, 3\}, 1)$-row $(\{1, 2, 4\}, 4)$-column entry is $1296n^{4} + 45738n^{3} - 43713n^{2} + 9828n$.
\item The $(\{1, 2, 3\}, 3)$-row $(\{1, 2, 4\}, 4)$-column entry is $1080n^{4} + 46008n^{3} - 43794n^{2} + 9828n$.
\item The $(\{1, 2, 3\}, 1)$-row $(\{1, 4, 5\}, 1)$-column entry is $1080n^{4} + 46224n^{3} - 43956n^{2} + 9828n$.
\item The $(\{1, 2, 3\}, 1)$-row $(\{1, 4, 5\}, 4)$-column entry is $756n^{4} + 46602n^{3} - 44064n^{2} + 9828n$.
\item The $(\{1, 2, 3\}, 2)$-row $(\{1, 4, 5\}, 4)$-column entry is $540n^{4} + 46872n^{3} - 44145n^{2} + 9828n$.
\item The $(\{1, 2, 3\}, 1)$-row $(\{4, 5, 6\}, 4)$-column entry is $47736n^{3} - 44496n^{2} + 9828n$.
\item The $(1, 2)$-row $(1, 2)$-column entry is $3672n^{4} + 42228n^{3} - 42390n^{2} + 9828n$.
\item The $(1, 2)$-row $(1, 3)$-column entry is $3132n^{4} + 43092n^{3} - 42741n^{2} + 9828n$.
\item The $(1, 2)$-row $(2, 1)$-column entry is $2592n^{4} + 43740n^{3} - 42930n^{2} + 9828n$.
\item The $(1, 2)$-row $(2, 3)$-column entry is $1296n^{4} + 45738n^{3} - 43713n^{2} + 9828n$.
\item The $(1, 2)$-row $(3, 1)$-column entry is $1296n^{4} + 45738n^{3} - 43713n^{2} + 9828n$.
\item The $(1, 2)$-row $(3, 2)$-column entry is $540n^{4} + 46872n^{3} - 44145n^{2} + 9828n$.
\item The $(1, 2)$-row $(3, 4)$-column entry is $47736n^{3} - 44496n^{2} + 9828n$.
\item The $(\hat{1}, \hat{2})$-row $(\hat{1}, \hat{1})$-column entry is $918n^{4} + 10503n^{3} - 10557n^{2} + 2457n$.
\item The $(\hat{1}, \hat{1})$-row $(\hat{1}, \hat{2})$-column entry is $918n^{4} + 10503n^{3} - 10557n^{2} + 2457n$.
\item The $(\hat{1}, \hat{1})$-row $(\hat{2}, \hat{2})$-column entry is $11934n^{3} - 11124n^{2} + 2457n$.
\item The $(\hat{1}, \hat{1})$-row $(\hat{2}, \hat{1})$-column entry is $648n^{4} + 10989n^{3} - 10773n^{2} + 2457n$.
\item The $(\hat{1}, \hat{1})$-row $(\hat{2}, \hat{3})$-column entry is $11934n^{3} - 11124n^{2} + 2457n$.
\item The $(\hat{1}, \hat{2})$-row $(\hat{1}, \hat{2})$-column entry is $810n^{4} + 10692n^{3} - 10638n^{2} + 2457n$.
\item The $(\hat{1}, \hat{2})$-row $(\hat{2}, \hat{1})$-column entry is $756n^{4} + 10800n^{3} - 10692n^{2} + 2457n$.
\item The $(\hat{1}, \hat{2})$-row $(\hat{1}, \hat{3})$-column entry is $540n^{4} + 11070n^{3} - 10773n^{2} + 2457n$.
\item The $(\hat{1}, \hat{2})$-row $(\hat{3}, \hat{1})$-column entry is $378n^{4} + 11367n^{3} - 10908n^{2} + 2457n$.
\item The $(\hat{1}, \hat{2})$-row $(\hat{2}, \hat{3})$-column entry is $378n^{4} + 11367n^{3} - 10908n^{2} + 2457n$.
\item The $(\hat{1}, \hat{2})$-row $(\hat{3}, \hat{2})$-column entry is $270n^{4} + 11556n^{3} - 10989n^{2} + 2457n$.
\item The $(\hat{1}, \hat{2})$-row $(\hat{3}, \hat{4})$-column entry is $11934n^{3} - 11124n^{2} + 2457n$.
\item The $(\{1, 2, 3\}, 1)$-row $(1, 2)$-column entry is $2376n^{4} + 44226n^{3} - 43173n^{2} + 9828n$.
\item The $(\{1, 2, 3\}, 2)$-row $(1, 2)$-column entry is $2052n^{4} + 44604n^{3} - 43281n^{2} + 9828n$.
\item The $(\{1, 2, 3\}, 3)$-row $(1, 2)$-column entry is $1836n^{4} + 44874n^{3} - 43362n^{2} + 9828n$.
\item The $(\{1, 2, 3\}, 1)$-row $(1, 4)$-column entry is $1836n^{4} + 45090n^{3} - 43524n^{2} + 9828n$.
\item The $(\{1, 2, 3\}, 2)$-row $(1, 4)$-column entry is $1296n^{4} + 45738n^{3} - 43713n^{2} + 9828n$.
\item The $(\{1, 2, 3\}, 1)$-row $(4, 1)$-column entry is $756n^{4} + 46602n^{3} - 44064n^{2} + 9828n$.
\item The $(\{1, 2, 3\}, 2)$-row $(4, 1)$-column entry is $540n^{4} + 46872n^{3} - 44145n^{2} + 9828n$.
\item The $(\{1, 2, 3\}, 1)$-row $(4, 5)$-column entry is $47736n^{3} - 44496n^{2} + 9828n$.
\item The $(\{1, 2, 3\}, 1)$-row $(\hat{1}, \hat{1})$-column entry is $1296n^{4} + 21978n^{3} - 21546n^{2} + 4914n$.
\item The $(\{1, 2, 3\}, 2)$-row $(\hat{1}, \hat{1})$-column entry is $918n^{4} + 22437n^{3} - 21681n^{2} + 4914n$.
\item The $(\{1, 2, 3\}, 1)$-row $(\hat{1}, \hat{2})$-column entry is $1134n^{4} + 22167n^{3} - 21600n^{2} + 4914n$.
\item The $(\{1, 2, 3\}, 2)$-row $(\hat{1}, \hat{2})$-column entry is $1080n^{4} + 22248n^{3} - 21627n^{2} + 4914n$.
\item The $(\{1, 2, 3\}, 3)$-row $(\hat{1}, \hat{2})$-column entry is $918n^{4} + 22437n^{3} - 21681n^{2} + 4914n$.
\item The $(\{1, 2, 3\}, 1)$-row $(\hat{1}, \hat{4})$-column entry is $756n^{4} + 22734n^{3} - 21816n^{2} + 4914n$.
\item The $(\{1, 2, 3\}, 2)$-row $(\hat{1}, \hat{4})$-column entry is $540n^{4} + 23004n^{3} - 21897n^{2} + 4914n$.
\item The $(\{1, 2, 3\}, 1)$-row $(\hat{4}, \hat{1})$-column entry is $540n^{4} + 23112n^{3} - 21978n^{2} + 4914n$.
\item The $(\{1, 2, 3\}, 2)$-row $(\hat{4}, \hat{1})$-column entry is $378n^{4} + 23301n^{3} - 22032n^{2} + 4914n$.
\item The $(\{1, 2, 3\}, 1)$-row $(\hat{4}, \hat{4})$-column entry is $23868n^{3} - 22248n^{2} + 4914n$.
\item The $(\{1, 2, 3\}, 1)$-row $(\hat{4}, \hat{5})$-column entry is $23868n^{3} - 22248n^{2} + 4914n$.
\item The $(1, 2)$-row $(\hat{1}, \hat{1})$-column entry is $2214n^{4} + 20547n^{3} - 20979n^{2} + 4914n$.
\item The $(1, 2)$-row $(\hat{1}, \hat{2})$-column entry is $1674n^{4} + 21303n^{3} - 21249n^{2} + 4914n$.
\item The $(1, 2)$-row $(\hat{1}, \hat{3})$-column entry is $1296n^{4} + 21870n^{3} - 21465n^{2} + 4914n$.
\item The $(1, 2)$-row $(\hat{2}, \hat{1})$-column entry is $1458n^{4} + 21681n^{3} - 21411n^{2} + 4914n$.
\item The $(1, 2)$-row $(\hat{2}, \hat{2})$-column entry is $918n^{4} + 22437n^{3} - 21681n^{2} + 4914n$.
\item The $(1, 2)$-row $(\hat{2}, \hat{3})$-column entry is $540n^{4} + 23004n^{3} - 21897n^{2} + 4914n$.
\item The $(1, 2)$-row $(\hat{3}, \hat{1})$-column entry is $918n^{4} + 22545n^{3} - 21762n^{2} + 4914n$.
\item The $(1, 2)$-row $(\hat{3}, \hat{2})$-column entry is $378n^{4} + 23301n^{3} - 22032n^{2} + 4914n$.
\item The $(1, 2)$-row $(\hat{3}, \hat{4})$-column entry is $23868n^{3} - 22248n^{2} + 4914n$.
\end{itemize}

\subsection{Entries of $U^4$}

The entries of $U^4$ can be obtained using identical methods used for the entries of $U^2$ and $U^3$. The entries of $U^4$ up to index type are:
\begin{itemize}
\item The $(\{1, 2, 3\}, 1)$-row $(\{1, 2, 3\}, 1)$-column entry is $22032n^{6} + 1506276n^{5} - 2246778n^{4} + 1093824n^{3} - 173664n^{2}$.
\item The $(\{1, 2, 3\}, 1)$-row $(\{1, 2, 3\}, 2)$-column entry is $21060n^{6} + 1507896n^{5} - 2247669n^{4} + 1093986n^{3} - 173664n^{2}$.
\item The $(\{1, 2, 3\}, 1)$-row $(\{1, 2, 4\}, 1)$-column entry is $16524n^{6} + 1519236n^{5} - 2257065n^{4} + 1096578n^{3} - 173664n^{2}$.
\item The $(\{1, 2, 3\}, 1)$-row $(\{1, 2, 4\}, 2)$-column entry is $15552n^{6} + 1520856n^{5} - 2257956n^{4} + 1096740n^{3} - 173664n^{2}$.
\item The $(\{1, 2, 3\}, 1)$-row $(\{1, 2, 4\}, 4)$-column entry is $13284n^{6} + 1525392n^{5} - 2260953n^{4} + 1097388n^{3} - 173664n^{2}$.
\item The $(\{1, 2, 3\}, 3)$-row $(\{1, 2, 4\}, 4)$-column entry is $11016n^{6} + 1529928n^{5} - 2263950n^{4} + 1098036n^{3} - 173664n^{2}$.
\item The $(\{1, 2, 3\}, 1)$-row $(\{1, 4, 5\}, 1)$-column entry is $11016n^{6} + 1532196n^{5} - 2267352n^{4} + 1099332n^{3} - 173664n^{2}$.
\item The $(\{1, 2, 3\}, 1)$-row $(\{1, 4, 5\}, 4)$-column entry is $7776n^{6} + 1538352n^{5} - 2271240n^{4} + 1100142n^{3} - 173664n^{2}$.
\item The $(\{1, 2, 3\}, 2)$-row $(\{1, 4, 5\}, 4)$-column entry is $5508n^{6} + 1542888n^{5} - 2274237n^{4} + 1100790n^{3} - 173664n^{2}$.
\item The $(\{1, 2, 3\}, 1)$-row $(\{4, 5, 6\}, 4)$-column entry is $1555848n^{5} - 2284524n^{4} + 1103544n^{3} - 173664n^{2}$.
\item The $(1, 2)$-row $(1, 2)$-column entry is $37584n^{6} + 1471284n^{5} - 2220210n^{4} + 1087020n^{3} - 173664n^{2}$.
\item The $(1, 2)$-row $(1, 3)$-column entry is $32076n^{6} + 1484244n^{5} - 2230497n^{4} + 1089774n^{3} - 173664n^{2}$.
\item The $(1, 2)$-row $(2, 1)$-column entry is $26568n^{6} + 1494936n^{5} - 2237382n^{4} + 1091232n^{3} - 173664n^{2}$.
\item The $(1, 2)$-row $(2, 3)$-column entry is $13284n^{6} + 1525392n^{5} - 2260953n^{4} + 1097388n^{3} - 173664n^{2}$.
\item The $(1, 2)$-row $(3, 1)$-column entry is $13284n^{6} + 1525392n^{5} - 2260953n^{4} + 1097388n^{3} - 173664n^{2}$.
\item The $(1, 2)$-row $(3, 2)$-column entry is $5508n^{6} + 1542888n^{5} - 2274237n^{4} + 1100790n^{3} - 173664n^{2}$.
\item The $(1, 2)$-row $(3, 4)$-column entry is $1555848n^{5} - 2284524n^{4} + 1103544n^{3} - 173664n^{2}$.
\item The $(\hat{1}, \hat{2})$-row $(\hat{1}, \hat{1})$-column entry is $9396n^{6} + 367254n^{5} - 554202n^{4} + 271431n^{3} - 43416n^{2}$.
\item The $(\hat{1}, \hat{1})$-row $(\hat{1}, \hat{2})$-column entry is $9396n^{6} + 367254n^{5} - 554202n^{4} + 271431n^{3} - 43416n^{2}$.
\item The $(\hat{1}, \hat{1})$-row $(\hat{2}, \hat{2})$-column entry is $388962n^{5} - 571131n^{4} + 275886n^{3} - 43416n^{2}$.
\item The $(\hat{1}, \hat{1})$-row $(\hat{2}, \hat{1})$-column entry is $6642n^{6} + 374301n^{5} - 560196n^{4} + 273132n^{3} - 43416n^{2}$.
\item The $(\hat{1}, \hat{1})$-row $(\hat{2}, \hat{3})$-column entry is $388962n^{5} - 571131n^{4} + 275886n^{3} - 43416n^{2}$.
\item The $(\hat{1}, \hat{2})$-row $(\hat{1}, \hat{2})$-column entry is $8262n^{6} + 370089n^{5} - 556551n^{4} + 272079n^{3} - 43416n^{2}$.
\item The $(\hat{1}, \hat{2})$-row $(\hat{2}, \hat{1})$-column entry is $7776n^{6} + 371466n^{5} - 557847n^{4} + 272484n^{3} - 43416n^{2}$.
\item The $(\hat{1}, \hat{2})$-row $(\hat{1}, \hat{3})$-column entry is $5508n^{6} + 376002n^{5} - 560844n^{4} + 273132n^{3} - 43416n^{2}$.
\item The $(\hat{1}, \hat{2})$-row $(\hat{3}, \hat{1})$-column entry is $3888n^{6} + 380214n^{5} - 564489n^{4} + 274185n^{3} - 43416n^{2}$.
\item The $(\hat{1}, \hat{2})$-row $(\hat{2}, \hat{3})$-column entry is $3888n^{6} + 380214n^{5} - 564489n^{4} + 274185n^{3} - 43416n^{2}$.
\item The $(\hat{1}, \hat{2})$-row $(\hat{3}, \hat{2})$-column entry is $2754n^{6} + 383049n^{5} - 566838n^{4} + 274833n^{3} - 43416n^{2}$.
\item The $(\hat{1}, \hat{2})$-row $(\hat{3}, \hat{4})$-column entry is $388962n^{5} - 571131n^{4} + 275886n^{3} - 43416n^{2}$.
\item The $(\{1, 2, 3\}, 1)$-row $(1, 2)$-column entry is $24300n^{6} + 1501740n^{5} - 2243781n^{4} + 1093176n^{3} - 173664n^{2}$.
\item The $(\{1, 2, 3\}, 2)$-row $(1, 2)$-column entry is $21060n^{6} + 1507896n^{5} - 2247669n^{4} + 1093986n^{3} - 173664n^{2}$.
\item The $(\{1, 2, 3\}, 3)$-row $(1, 2)$-column entry is $18792n^{6} + 1512432n^{5} - 2250666n^{4} + 1094634n^{3} - 173664n^{2}$.
\item The $(\{1, 2, 3\}, 1)$-row $(1, 4)$-column entry is $18792n^{6} + 1514700n^{5} - 2254068n^{4} + 1095930n^{3} - 173664n^{2}$.
\item The $(\{1, 2, 3\}, 2)$-row $(1, 4)$-column entry is $13284n^{6} + 1525392n^{5} - 2260953n^{4} + 1097388n^{3} - 173664n^{2}$.
\item The $(\{1, 2, 3\}, 1)$-row $(4, 1)$-column entry is $7776n^{6} + 1538352n^{5} - 2271240n^{4} + 1100142n^{3} - 173664n^{2}$.
\item The $(\{1, 2, 3\}, 2)$-row $(4, 1)$-column entry is $5508n^{6} + 1542888n^{5} - 2274237n^{4} + 1100790n^{3} - 173664n^{2}$.
\item The $(\{1, 2, 3\}, 1)$-row $(4, 5)$-column entry is $1555848n^{5} - 2284524n^{4} + 1103544n^{3} - 173664n^{2}$.
\item The $(\{1, 2, 3\}, 1)$-row $(\hat{1}, \hat{1})$-column entry is $13284n^{6} + 748602n^{5} - 1120392n^{4} + 546264n^{3} - 86832n^{2}$.
\item The $(\{1, 2, 3\}, 2)$-row $(\hat{1}, \hat{1})$-column entry is $9396n^{6} + 756216n^{5} - 1125333n^{4} + 547317n^{3} - 86832n^{2}$.
\item The $(\{1, 2, 3\}, 1)$-row $(\hat{1}, \hat{2})$-column entry is $11664n^{6} + 751680n^{5} - 1122336n^{4} + 546669n^{3} - 86832n^{2}$.
\item The $(\{1, 2, 3\}, 2)$-row $(\hat{1}, \hat{2})$-column entry is $11016n^{6} + 753138n^{5} - 1123389n^{4} + 546912n^{3} - 86832n^{2}$.
\item The $(\{1, 2, 3\}, 3)$-row $(\hat{1}, \hat{2})$-column entry is $9396n^{6} + 756216n^{5} - 1125333n^{4} + 547317n^{3} - 86832n^{2}$.
\item The $(\{1, 2, 3\}, 1)$-row $(\hat{1}, \hat{4})$-column entry is $7776n^{6} + 760428n^{5} - 1128978n^{4} + 548370n^{3} - 86832n^{2}$.
\item The $(\{1, 2, 3\}, 2)$-row $(\hat{1}, \hat{4})$-column entry is $5508n^{6} + 764964n^{5} - 1131975n^{4} + 549018n^{3} - 86832n^{2}$.
\item The $(\{1, 2, 3\}, 1)$-row $(\hat{4}, \hat{1})$-column entry is $5508n^{6} + 766098n^{5} - 1133676n^{4} + 549666n^{3} - 86832n^{2}$.
\item The $(\{1, 2, 3\}, 2)$-row $(\hat{4}, \hat{1})$-column entry is $3888n^{6} + 769176n^{5} - 1135620n^{4} + 550071n^{3} - 86832n^{2}$.
\item The $(\{1, 2, 3\}, 1)$-row $(\hat{4}, \hat{4})$-column entry is $777924n^{5} - 1142262n^{4} + 551772n^{3} - 86832n^{2}$.
\item The $(\{1, 2, 3\}, 1)$-row $(\hat{4}, \hat{5})$-column entry is $777924n^{5} - 1142262n^{4} + 551772n^{3} - 86832n^{2}$.
\item The $(1, 2)$-row $(\hat{1}, \hat{1})$-column entry is $22680n^{6} + 726894n^{5} - 1103463n^{4} + 541809n^{3} - 86832n^{2}$.
\item The $(1, 2)$-row $(\hat{1}, \hat{2})$-column entry is $17172n^{6} + 738720n^{5} - 1112049n^{4} + 543915n^{3} - 86832n^{2}$.
\item The $(1, 2)$-row $(\hat{1}, \hat{3})$-column entry is $13284n^{6} + 747468n^{5} - 1118691n^{4} + 545616n^{3} - 86832n^{2}$.
\item The $(1, 2)$-row $(\hat{2}, \hat{1})$-column entry is $14904n^{6} + 744390n^{5} - 1116747n^{4} + 545211n^{3} - 86832n^{2}$.
\item The $(1, 2)$-row $(\hat{2}, \hat{2})$-column entry is $9396n^{6} + 756216n^{5} - 1125333n^{4} + 547317n^{3} - 86832n^{2}$.
\item The $(1, 2)$-row $(\hat{2}, \hat{3})$-column entry is $5508n^{6} + 764964n^{5} - 1131975n^{4} + 549018n^{3} - 86832n^{2}$.
\item The $(1, 2)$-row $(\hat{3}, \hat{1})$-column entry is $9396n^{6} + 757350n^{5} - 1127034n^{4} + 547965n^{3} - 86832n^{2}$.
\item The $(1, 2)$-row $(\hat{3}, \hat{2})$-column entry is $3888n^{6} + 769176n^{5} - 1135620n^{4} + 550071n^{3} - 86832n^{2}$.
\item The $(1, 2)$-row $(\hat{3}, \hat{4})$-column entry is $777924n^{5} - 1142262n^{4} + 551772n^{3} - 86832n^{2}$.
\end{itemize}

\subsection{The minimal polynomial of $U$}

Fix $n \geq 4$. Let $b = -3n(4n-3)$ and $c = 54[2\binom{n+2}{4} + 5\binom{n+1}{4} + \binom{n}{4}]$ and $d=15n(2n-1)$. Note $c=18n^4-27n^3+9n^2$. Let $f=b-d$ and $g=c-bd$ and $h=-cd$. Note $f = -42n^2+24n$ and $g = 378n^4 - 477n^3 + 144n^2$ and $h = -540n^6 + 1080n^5 - 675n^4 + 135n^3$. 

Define $\mu(x) = x(x-d)(x^2+bx+c)$. Then $\mu(x)=x^4+fx^3+gx^2+hx$. For example, when $n=4$, then $\mu(x) = x(x-420)(x^2 - 156x + 3024) = x^4 - 576x^3 + 68544x^2 - 1270080x$, and when $n=5$, then $\mu(x) = x(x-675)(x^2 - 255x + 8100) = x^4 - 930x^3 + 180225x^2 - 5467500x$.

It turns out that $U^4 + fU^3 + gU^2 + hU = 0$. For instance, the $(\{1, 2, 3\}, 1)$-row $(\{1, 2, 3\}, 1)$-column entry of $U^4 + fU^3 + gU^2 + hU$ is $(22032n^{6} + 1506276n^{5} - 2246778n^{4} + 1093824n^{3} - 173664n^{2}) + (-42n^2+24n)(2160n^{4} + 44496n^{3} - 43254n^{2} + 9828n) + (378n^4 - 477n^3 + 144n^2)(216n^{2} + 1026n - 432) + (-540n^6 + 1080n^5 - 675n^4 + 135n^3)(24)$, which is equal to $0$. A similar calculation has been verified for every entry type, which the reader may check using the provided code. Thus, $\mu(U) = 0$.

We note $f<0$ and $g>0$ and $h<0$, since $b<0$ and $c>0$ and $d>0$. By Descartes' Law of Signs, $\mu$ has no negative real roots. Thus $\mu$ is (a multiple of) the minimum polynomial of $U$. Since $U$ is symmetric (and thus all eigenvalues of $U$ are real), this proves that $U$ is positive semidefinite. (For $n \leq 3$, the positive semidifefiniteness of $U$ was directly verified.)

\subsection{Remarks}

Since $\mu$ is (a multiple of) the minimum polynomial of $U$, the distinct eigenvalues of $U$ are $0$, the integer $d=15n(2n-1)$, and the roots of $q(x) = x^2+bx+c$. While our proof only shows that $\mu$ is a multiple of the minimum polynomial of $U$, but through computation for many values of $n$, we believe that $\mu$ is the minimum polynomial of $U$.

While the linear coefficient of $q(x)$ had a relatively nice formula $b = -3n(4n-3)$, the constant term $c = 54[2\binom{n+2}{4} + 5\binom{n+1}{4} + \binom{n}{4}]$ deserves some explanation. The values of $\frac{1}{54}c$ starting at $n=4$ are $56$, $150$, $330$, $637$, $1120$, $1836$, and so on. This sequence is every other term of A181474 (see~\cite{oeis:A181474}). Thus, $\frac{1}{54}c$ is the coefficient of $x^{2n-3}$ in the generating function
\[\gamma(x) = \frac{1+x+4x^2+x^3+x^4}{(1-x)^5(1+x)^4}.\]
After applying Taylor expansion to $\frac1{(1-x^2)^4}$ and then to $\frac1{(1-x^2)^4(1-x)}$, the result of multiplying by $1+x+4x^2+x^3+x^4$ yielded that the coefficient of $x^{2n-3}$ in $\gamma(x)$, and thus the value of $\frac{1}{54}c$, is the following weighted sum of binomial coefficients:
\[2\binom{n+2}{4} + 5\binom{n+1}{4} + \binom{n}{4}.\]

Though we already showed that $U$ is positive-semidefinite in the previous subsection via alternating coefficients on the minimal polynomial, we have computational evidence which suggests the following additional spectral data:

First, recall that $U$ is a square matrix of size $3\binom{n}{3} + n(n-1) + n^2$. It appears that the integer $d$ is an eigenvalue of $U$ with multiplicity $1$, while $0$ is an eigenvalue of multiplicity $w = 3\binom{n}{3} + 2n^2 - 3n+1$ and the two distinct roots of $q(x)$ appear as eigenvalues of multiplicity $n-1$ each. That is to say, it appears that the characteristic polynomial of $U$ is $\chi(x) = x^w(x-d)(x^2+bx+c)^{n-1}$.

\subsubsection{Eigenvectors of $U$}

Let $e_{(\{i,j,k\},i)}$ denote the standard unit basis vector for the $(\{i,j,k\},i)$-entry in block $1$.
Likewise, $e_{(i,j)}$ denotes the standard basis vector in block $2$, and $e_{(\hat{i},\hat{j})}$ denotes the standard basis vector in block $3$.

For a fixed $n$, the vector
\[ \sum_{(\{i,j,k\},i)} e_{(\{i,j,k\},i)} + \sum_{(i,j)} e_{(i,j)} + \frac12 \sum_{(\hat{i},\hat{j})} e_{(\hat{i},\hat{j})} \]
is an eigenvector of $U$ corresponding to $\lambda = d$.

We present below a set of linearly independent eigenvectors of $U$ corresponding to $\lambda = 0$:
\begin{itemize}
\item $e_{(\{i,j,k\},i)}-e_{(\hat{j},\hat{n})}-e_{(\hat{k},\hat{n})}-2e_{(\hat{n},\hat{i})}+2e_{(\hat{n},\hat{n})}$ whenever $i,j,k\neq n$
\item $e_{(i,j)}-e_{(\hat{i},\hat{n})}-e_{(\hat{j},\hat{n})}-2e_{(\hat{n},\hat{i})}+2e_{(\hat{n},\hat{n})}$ whenever $i,j\neq n$
\item $e_{(\{i,j,n\},i)}-e_{(\hat{i},\hat{n})}-e_{(\hat{j},\hat{n})}-2e_{(\hat{n},\hat{j})}+e_{(\hat{n},\hat{n})}$ whenever $i,j\neq n$
\item $e_{(i,n)}-e_{(\hat{i},\hat{n})}-2e_{(\hat{n},\hat{i})}+e_{(\hat{n},\hat{n})}$ whenever $i\neq n$
\item $e_{(\{i,j,n\},n)}-e_{(\hat{i},\hat{n})}-e_{(\hat{j},\hat{n})}$ whenever $i,j\neq n$
\item $e_{(n,i)}-e_{(\hat{i},\hat{n})}-e_{(\hat{n},\hat{n})}$ whenever $i\neq n$
\item $e_{(\hat{i},\hat{j})} - e_{(\hat{i},\hat{n})} - e_{(\hat{n},\hat{j})} + e_{(\hat{n},\hat{n})}$ whenever $i,j \leq n$
\end{itemize} 
Based on the coordinates of the seven types of vectors mentioned above, this gives 
$3\binom{n-1}{3} + 2\binom{n-1}{2} + 2\binom{n-1}{2} + (n-1) + \binom{n-1}{2} + (n-1) + (n-1)^2$ of linearly independent eigenvectors of $U$ corresponding to $\lambda = 0$, thus the geometric multiplicity of $\lambda=0$ matches the algebraic multiplicity $w = 3\binom{n}{3} + 2n^2 - 3n+1$.

Diagonalization for each $n$ would be complete if we could, for each $n$, obtain exact representations of $n-1$ linearly independent eigenvectors of $U$ for each root of $q(x)=x^2+bx+c$. The sum of the eigenspaces corresponding to the two roots of $q$ form a $2(n-1)$-dimensional subspace of $\mathbb{R}^{3\binom{n}{3} + n(n-1) + n^2}$. Say that $c_1, \dots, c_{2(n-1)}$ is a basis for this subspace. We make use of the fact that each $c_i$ is orthogonal to the eigenvectors corresponding to $\lambda=0$ and $\lambda=d$. Thus, a concrete way to find obtain a basis is to place the eigenvectors for $0$ and $d$ as the rows of a matrix and find a basis for its kernel. We denote by $V$ the $(3\binom{n}{3} + n(n-1) + n^2) \times 2(n-1)$ matrix whose columns are the vectors $c_1, \dots, c_{2(n-1)}$ generating the kernel. Fix $\lambda$ to be a root of $q(x)$. Then finding a non-zero $x \in \mathbb{R}^{3\binom{n}{3} + n(n-1) + n^2}$ satisfying $Ux=\lambda x$ is the same as finding a parameter vector $\alpha \in \mathbb{R}^{2(n-1)}$ satisfying $W\alpha = 0$, where $x = V\alpha$ and $W=UV-\lambda V$. Finding an exact representation for a vector $\alpha$ in the kernel of $UV-\lambda V$ can be done for fixed $n$, and then $x=V\alpha$ is an eigenvector of $U$ corresponding to $\lambda$.

For concreteness, we describe this process for $n=4$. In this case, $U$ is a $40 \times 40$ matrix, and the roots of $q(x)=x^2 - 156x + 3024$ are $78 \pm 6 \sqrt{85}$. Let us fix $\lambda = 78 + 6 \sqrt{85}$. A minimal generating set of eigenvectors for $0$ and $d$ consists of $34$ vectors. These vectors can be the rows of a $34 \times 40$ matrix, and a basis for the kernel of this matrix is found as the columns of
\[
V = \left[\begin{array}{rrrrrr}
1 & 0 & 0 & 0 & 0 & 0 \\
0 & 1 & 0 & 0 & 0 & 0 \\
0 & 0 & 1 & 0 & 0 & 0 \\
0 & 0 & 0 & 1 & 0 & 0 \\
-1 & 1 & 0 & 1 & 0 & 0 \\
0 & 0 & 0 & 0 & 1 & 0 \\
0 & 0 & 0 & 0 & 0 & 1 \\
0 & 1 & -1 & 0 & 0 & 1 \\
1 & 0 & -1 & -1 & 1 & 1 \\
-1 & -1 & 1 & 0 & -1 & -1 \\
0 & -1 & 0 & 0 & -1 & -1 \\
0 & -1 & 0 & -1 & 0 & -1 \\
0 & 2 & 0 & 1 & 1 & 1 \\
1 & 2 & -1 & 0 & 1 & 2 \\
0 & 2 & -1 & 1 & 1 & 2 \\
0 & 0 & 1 & 1 & 0 & -1 \\
0 & -1 & 1 & 0 & -1 & -2 \\
-1 & -1 & 1 & 1 & -1 & -2 \\
1 & 0 & 0 & -1 & 0 & 1 \\
0 & -1 & 1 & -1 & -1 & -1 \\
0 & -1 & 0 & -1 & -1 & 0 \\
0 & 0 & -1 & 0 & 1 & 1 \\
-1 & -1 & 0 & 0 & 0 & -1 \\
0 & -1 & -1 & -1 & 0 & 0 \\
\frac{1}{2} & \frac{3}{2} & -\frac{1}{2} & \frac{1}{2} & 1 & \frac{3}{2} \\
\frac{1}{2} & \frac{1}{2} & 0 & \frac{1}{2} & \frac{1}{2} & \frac{1}{2} \\
\frac{1}{2} & \frac{1}{2} & 0 & 0 & \frac{1}{2} & 1 \\
\frac{1}{2} & \frac{1}{2} & -\frac{1}{2} & 0 & 1 & 1 \\
-\frac{1}{2} & \frac{1}{2} & \frac{1}{2} & \frac{1}{2} & 0 & -\frac{1}{2} \\
-\frac{1}{2} & -\frac{1}{2} & 1 & \frac{1}{2} & -\frac{1}{2} & -\frac{3}{2} \\
-\frac{1}{2} & -\frac{1}{2} & 1 & 0 & -\frac{1}{2} & -1 \\
-\frac{1}{2} & -\frac{1}{2} & \frac{1}{2} & 0 & 0 & -1 \\
\frac{1}{2} & \frac{1}{2} & -\frac{1}{2} & -\frac{1}{2} & 0 & \frac{1}{2} \\
\frac{1}{2} & -\frac{1}{2} & 0 & -\frac{1}{2} & -\frac{1}{2} & -\frac{1}{2} \\
\frac{1}{2} & -\frac{1}{2} & 0 & -1 & -\frac{1}{2} & 0 \\
\frac{1}{2} & -\frac{1}{2} & -\frac{1}{2} & -1 & 0 & 0 \\
-\frac{1}{2} & \frac{1}{2} & -\frac{1}{2} & \frac{1}{2} & 0 & \frac{1}{2} \\
-\frac{1}{2} & -\frac{1}{2} & 0 & \frac{1}{2} & -\frac{1}{2} & -\frac{1}{2} \\
-\frac{1}{2} & -\frac{1}{2} & 0 & 0 & -\frac{1}{2} & 0 \\
-\frac{1}{2} & -\frac{1}{2} & -\frac{1}{2} & 0 & 0 & 0
\end{array}\right]
.\]
Let $W = UV - (78 + 6 \sqrt{85})V$. Instead of describing a complete basis for the eigenspace of $78 + 6 \sqrt{85}$, we summarize what happens with one vector. For any vector $\alpha$ satisfying $W\alpha=0$, the vector $x = V \alpha$ is an eigenvector of $U$ corresponding to $\lambda = 78 + 6\sqrt{85}$ given with exact coordinates. Displaying $\alpha$ and $x$ requires large-format paper, and both are available at
\begin{center}
\url{https://www.overleaf.com/read/xhmyxkywrqhb}
\end{center}
While we have an alternate proof that $U$ is positive semidefinite, what remains open (and interesting) is finding, for all $n$, exact representations of eigenvectors corresponding to both roots of $q(x)$.

\section*{Acknowledgements}

The authors are grateful for support from the Dean's Distinguished Fellowship of the College of Science and Health at the University of Wisconsin-La Crosse. The authors are also grateful for support from a Undergraduate Research \& Creativity Committee Grant from the University of Wisconsin-La Crosse.

\end{document}